\documentclass[11pt]{article}
\usepackage{graphicx} 
\usepackage[left=1.25in,top=1in,right=1.25in,bottom=1.00in]{geometry}

\usepackage{cite}
\usepackage{amsmath}
\usepackage{amssymb}
\usepackage{amsbsy}
\usepackage{amsthm}
\usepackage{diagbox} 
\usepackage{array}
\usepackage{hyperref}
\usepackage{epsfig}
\usepackage{color}
\usepackage{float}
\usepackage{multirow}
\usepackage{enumitem}
\usepackage{booktabs} 
\usepackage{multirow} 
\newtheorem{theorem}{Theorem}[section]

\newtheorem{proposition}[theorem]{Proposition}

\usepackage[table]{xcolor}
\usepackage{authblk}  

\title{Information Asymmetry in Queues with Strategic Customers}

\author[1]{Shunan Zheng\thanks{*Corresponding author. Email: sn.zheng@utexas.edu}}
\author[1]{John Hasenbein}
\affil[1]{\small Operations Research and Industrial Engineering, The University of Texas at Austin, 204 E. Dean Keeton St., Austin, TX 78712, United States}

\begin{document}
\maketitle

\begin{abstract}
This paper studies information asymmetry in an unobservable single-server queueing system. While system managers have knowledge of the true arrival rate, customers may lack this information and instead form arbitrary beliefs. We propose a three-tier hierarchy of information asymmetry with increasing levels of information disclosure: customers keep private beliefs, customers are aware of the beliefs of others, and customers know the true arrival rate. Within this framework, the effects of the belief distribution, which is assumed to be general with minimal restrictions, are analyzed in terms of equilibrium joining probabilities, revenue, and social welfare. Furthermore, strategies for information disclosure are proposed for system managers to regulate the queue.

\end{abstract}

\section{Background}

\subsection{Introduction}

In the classical strategic queueing literature, a common assumption is that customers have full knowledge of the system \cite{naor_regulation_1969,edelson_congestion_1975}. In reality, this is often not the case, in that customers usually have more limited access to information than managers, and their information levels are highly subjective to what the managers choose to reveal. This information asymmetry creates uncertainty for customers, who must factor incomplete information into their decision-making. 

In recent years, there has been a growing body of research dedicated to exploring information asymmetry and parameter uncertainty in queueing models. While the former stems from the latter, it is useful to study the effects of limited customer awareness separately. In this paper, we focus primarly on information asymmetry. In addition, we focus specifically on the asymmetry of the arrival rate, although our approach can be extended to other system parameters, such as service rate, payoff, and waiting cost. 

A direct consequence of information asymmetry is that customers must develop their own beliefs about the arrival rate to evaluate their waiting time. Diverse past experiences lead to heterogeneous beliefs about the arrival rate, which collectively form the customer belief distribution. We discuss this belief more thoroughly in Section \ref{Chap: Belief}. Customers who believe that the arrival rate is low are optimistic, expecting a less congested system and a better experience, while those who believe the arrival rate is higher are pessimistic, anticipating more congestion and a worse experience. Our paper differs from existing literature in that we study a general form of belief distribution, which makes our conclusions more adaptable to various scenarios.

Some third-party platforms have recognized the importance of reducing information gaps between system managers and customers. Typically, they provide two types of data: historical estimates of the expected congestion and live updates on the current system state. The former provides less precise information than the latter, but it is often better than having no information. Inspired by this setting, we categorize the information availability into three levels, which are discussed in Section \ref{Chap: Info Levels} and serve as a powerful tool to approach this problem. 
In this paper, uncertainty regarding system congestion is manifested through the unknown arrival rate. 

Based on this framework, we explore how the distribution of customer beliefs and the level of information revealed by managers impact the system in terms of the equilibrium joining probability, revenue, and social welfare. We also provide managerial insight for making joint decisions on both the fee to join the system and the disclosure of information to customers.

\subsection{Related Literature}

 Naor's paper \cite{naor_regulation_1969} is the first to study the strategic customer behavior in queues, which introduced the concept of customers as active decision-makers who choose whether to join a queue based on observable queue lengths. Following Naor's paper, the unobservable queue was first studied by Edelson and Hildebrand \cite{edelson_congestion_1975}. Since the publication of these foundational papers, there has developed a large body of literature on related models.
At the time of its publication, Hassin's book reviewed essentially every related paper \cite{hassin_rational_2016}. 

More recent papers investigate the idea of parameter uncertainty in strategic queues. Hassin and Snitkovsky \cite{hassin_social_2020} provide
conditions under which "Naor's Inequality" holds, when there is
uncertainty in the arrival rate to the system. In a closely related paper, Chen and Hasenbein \cite{chen_knowledge_2020} show that under arrival rate uncertainty, the decisions of the social optimizer and revenue maximizer no longer coincide. Hassin et al.\cite{hassin_strategic_2023} investigate the same type of system under different assumptions on customer rationality. When the expected service time and value of service are unknown but positively correlated, Debo and Veeraraghavan \cite{debo_equilibrium_2014} show that the equilibrium structure involves non-monotone queue joining. 
\cite{cohen_learning-based_2024} propose a learning-based dispatching algorithm for an M/M/1 queuing system with unknown service and arrival rates. 

Some papers focus on designing queueing mechanisms for queue regulation. \cite{economou2021impact} discuss several techniques for the control of information in a given system and their impact. A hybrid information structure between observable and unobservable periods, which potentially increases equilibrium throughput and social welfare, is studied in \cite{economou2021impact, dimitrakopoulos_strategic_2021}. Delay or queue-related information shared with customers influences their strategic behavior and overall system performance as noted in\cite{burnetas_strategic_2017,ibrahim2017does}. Haviv and Oz \cite{haviv_self-regulation_2018} propose regulation schemes, under which the resulting equilibrium joining rate coincides with the social optimum. 

Other related work considers information disclosure as a regulation tool, studying how customers belief updates under information asymmetry. The work of Hassin \cite{hassin_information_2007} measures how uncertainty affects profits in an unobservable single-server queue, and whether the manager has an incentive to reveal the information under different pricing strategies.  Guo and Haviv \cite{guo_optimal_2022} explore the optimal disclosure of queue length information when service quality is uncertain. Ravner and Sakuma \cite{ravner_strategic_2021} study the strategic choice of arrival time with service time uncertainty, considering different levels of customer rationality and knowledge. While these papers assume a simple two-class belief structure, Cui and Veeraraghavan \cite{cui_blind_2016} study the impact of arbitrary beliefs about the service rate in an observable queue, demonstrating that information disclosure does not always improve social welfare and may even reduce consumer welfare. 

In this paper, we assume deterministic system parameters. When the system parameter is revealed to the customers, our framework matches \cite{edelson_congestion_1975}, which we refer to as the classical case. However, if this information is not disclosed, the customers incorporate uncertainty into their decision-making due to not knowing its true value. Our work follows the assumption of a general belief distribution as in \cite{chen_knowledge_2020} and \cite{cui_blind_2016} under the unobservable case and considers the impact of information disclosure.
 
The rest of the paper is structured as follows. Section \ref{Sec:Model} introduces our model setup and key assumptions regarding customers' beliefs and the information state hierarchy. For the reader’s convenience, we summarize our key results in Section \ref{Chap: notations}. Section \ref{Sec: derivations} derives the equilibrium, revenue, and social welfare of each information state. Section \ref{Sec: Impact} analyzes the impact of information asymmetry from three perspectives, covering the equilibrium in Section \ref{Chap: ImpactEq}, revenue in Section \ref{Chap: ImpactRev}, and the social welfare in Section \ref{Chap: ImpactSW}. In particular, we compare the values under fixed price or joining rate, and also the maximum achievable values by optimizing over price. Our focus is on the relationships between these factors rather than on their absolute values. The conclusions in Section \ref{Sec: Impact} provide insights into the effects of information disclosure in Section \ref{Sec: Insights}. Finally, in Section \ref{Sec: Numerics}, we illustrate our results with numerical examples.

\section{Model} \label{Sec:Model}

We consider a single-server M/G/1 queueing system with Poisson arrivals at rate $\lambda$. Suppose that the service time \(S\) of each customer follows a general distribution that satisfies
\[
\mathbb{E}[S] = \frac{1}{\mu}, \quad \mathbb{E}[S^2] = s^2.
\]
Define the traffic intensity as $\rho = \frac{\lambda}{\mu}, \text{with } \rho < 1$. Then, according to the Pollaczek-Khinchine formula, the expected waiting time in the system is:
\[
\mathbb{E}[W] = \frac{\lambda s^2}{2(1 - \rho)}+\frac{1}{\mu}.
\]
For the special case of an M/M/1 queue, the service time follows exponential distribution with rate $\mu$ and the expected waiting time in the system is \(\mathbb{E}[W] = \frac{1}{\mu-\lambda}\).

Customers receive a payoff \(R\) upon completing the service and incur a waiting cost \(C\) per unit of waiting time. The manager charges a service fee \(p\) to each customer. The customer chooses a strategy to join the queue with probability \(q\). Denote the equilibrium expected waiting time by \(W(\lambda_e)\), where \(\lambda_e\) is the effective arrival rate. A customer’s expected utility is given by \(R - p - CW(\lambda_e)\) if they join the queue and zero if they balk. Throughout this paper, we assume that \(R\geq \frac{C}{\mu} \). Otherwise, no customer would ever join the queue, since the customer’s utility would be \(R-p-C\,W(\lambda_e)\leq R-\frac{C}{\mu}<0\). We study the unobservable setting, in which customers cannot observe the number of customers in the system before deciding to join. We also assume that customers may not renege.

Customers do not know the arrival rate \(\lambda\). Rather, each customer forms a belief about the arrival rate. In general, we allow these beliefs to be different for each customer, and we thus denote these as heterogeneous beliefs. However, the customers do know the service rate and the other system parameters. Based on this knowledge and the arrival rate belief, arriving customers calculates their expected waiting time based on their belief, joining only if their expected net benefit is non-negative. 

In contrast, we assume that the manager has full knowledge of system parameters, including the arrival rate. The manager also knows the distribution of customer arrival rate beliefs, but she does not know the belief of any particular customer who arrives. We consider two types of managers: (1) the revenue maximizer (RM), who seeks to maximize total fees collected from customers, and (2) the social optimizer (SO), who aims to maximize total utility. As usual, the service fee $p$ is considered a transfer payment in the case of a social optimizer. 

\subsection{The Unobservable Queueing Game}

To identify the information flow and to avoid ambiguity in our problem formulation, we specify the order the queueing game is played:

\begin{itemize}
    \item[] \textbf{(Nature)} At the beginning of each day, ``nature'' selects a certain arrival rate \(\lambda\) for the system, which may vary from day to day but remains constant within a single day. 
    \item[] \textbf{(Manager)} The manager is able to observe this true arrival rate \(\lambda\). Knowing the arrival rate belief of the customers and all the system parameters, the manager then sets the entrance fee for the day. All these actions are taken in a negligible amount of time, before the arrival of any customer. The manager is not allowed to change the service fee once it is set.
    \item[] \textbf{(Customers)}  Each arriving customer chooses a joining probability \(q\) that induces a non-negative net benefit, based on their belief regarding the arrival rate. Then, they flip a coin with a winning probability of \(q\) to decide whether or not to join the queue. They cannot renege once they join the queue.
\end{itemize}

The motivation for allowing nature to select an arrival rate each day rather than fix on a certain arrival rate is to justify the formation of the heterogeneous belief. Otherwise, the customers will gradually learn the true arrival rate over time. 

\subsection{The Arrival Rate Belief} \label{Chap: Belief}

Customers form their beliefs about the arrival rate based on past experiences, word of mouth, media, and other information sources. However, these sources can only partially depict the system state, not to mention the uncertainty in future system parameters, which even the manager cannot fully anticipate.

These individual beliefs collectively form an arrival rate belief, which we denote by the non-degenerate random variable \(\Lambda\). The arrival rate belief \(\Lambda\) follows a general non-negative distribution. If the distribution is continuous, then we denote the density by \( f_{\Lambda}(\lambda) \). To clearly distinguish the customers' arrival rate belief \(\Lambda\) from the arrival rate \(\lambda\) determined by nature, we will refer to \(\lambda\) as the true arrival rate from now on. Note that our model does not exactly represent parameter uncertainty, even though \(\Lambda\) represents arrival rate uncertainty in some other related work. Note that the deterministic \(\lambda\) here indicates that the uncertainty in our model is not brought about by the system parameter itself, unlike in  some other related work.
Rather, it is brought about by \(\Lambda\), which represents the spectrum of customers' (individually degenerate) beliefs.

Denote the support of \(\Lambda\) by \([\lambda_{min},\lambda_{max}]\), which is a subset of \([0, \bar{\lambda}]\) and satisfies \(\lambda_{min}<\lambda<\lambda_{max}\), where \(\lambda\) is the true arrival rate. In particular, \(\bar{\lambda}\) represents the maximum possible arrival rate belief. The boundedness condition is primarily required to justify interchange operations in the proofs below. We let $\bar{\lambda} = \mu$ in this paper without loss of generality. If $\bar{\lambda}$ is greater than $\mu$, then the system manager (RM or SO) will set the fee $p$ to guarantee that the effective arrival rate does not exceed the server's processing rate. We do not make any assumptions about belief distribution of \(\Lambda\) beyond the finite support and that the range of \(\Lambda\) should cover \(\lambda\). Thus, the conclusions in our paper are general and can apply to any form of discrete or continuous distribution. 

We assume that there will be no return customers that arrive to the system more than once per day, so that no customers will come back with an updated belief. Still, it is possible that fully rational customers could infer the true arrival rate upon observing $p$. However, such an inference requires complex calculations that may be impractical for customers to perform at the moment of arrival, unlike the straightforward computation of their own benefit. Hence, we assume customers have limitations that prevent them from making fully rational decisions that update their beliefs.

\subsection{Three Levels of Information Asymmetry} \label{Chap: Info Levels}

Our main objective is to study the effect of information asymmetry in strategic queueing systems. The asymmetry, as alluded to earlier, is between the manager, who has full system knowledge, and the customers, who do not know the actual arrival rate. If all customers do know the true arrival rate, then this scenario is equivalent to the classical case studied by \cite{edelson_congestion_1975}. In this paper, we introduce two additional asymmetry scenarios, inducing a hierarchy of three cases with decreasing levels of customer information. We elaborate on each case in detail below, starting with the classical case and progressively reducing the information available to customers as they formulate their strategies.

\begin{itemize}

    \item[] \textbf{Classical Case:} Customers know the true arrival rate and are aware that all other customers also know the true arrival rate. However, they do not observe the queue length upon arrival, a condition that also applies to the first and second levels. 

    \item[] \textbf{Shared Belief Case:} Customers are aware of the belief distribution over the arrival rate, i.e., the distribution of $\Lambda$, which represents the joint belief distribution. They obtain this information either through an announcement by the manager or via a belief-sharing mechanism.
    
    \item[] \textbf{Private Belief Case:} Customers do not know the true arrival rate and instead form heterogeneous beliefs about it. They neither share their beliefs with other customers nor receive this information from the manager. Moreover, customers are unaware of their position within the belief distribution and thus assume that others share the same belief. 

\end{itemize}

We focus on comparing three levels of information asymmetry regarding the arrival rate. The information hierarchy allows us to examine the impact of information step by step, breaking down the actions we can take to improve either revenue or social welfare.

\subsection{Notation and Main Contributions} \label{Chap: notations}

Detailed results and proofs are discussed in later sections. Here, we summarize our main results, and the notation used in the remainder of the paper. One object of study is the 
customers' joining probability. In the classical and shared belief cases, there is one joining probability for all customers. In the private belief case, this probability depends on each individual customer's belief regarding the arrival rate, thus varies between individuals. For quantities that depend on the customer's private belief, we
use $\tilde{\lambda}$ to denote the dependence on this belief. 
We use $q$, with various superscripts and subscripts, to denote this quantity. The net benefit rate to the customers is denoted by $\text{U}$. In addition, we study system-wide metrics such as the revenue rate and social welfare rate. These are denoted by Rev and SW, respectively. Table \ref{Notations} summarizes the specific notation used in the various cases. 

\begin{table}[h]
    \centering
    \caption{Summary of Notation. We will give formal definitions of these terms in Section \ref{Sec: derivations}.}
    \begin{tabular}{l|c|c c c}
        \midrule
        & Manager  & Classical & Shared Belief & Private Belief \\
        \midrule
        \multirow{3}{*}{Equilibrium Joining Probability} & Individual & $q^C_e$ & $q^S_e$& $q_e^P(\tilde{\lambda})$ \\
                                      & RM & $q^C_m$ & $q^S_m$ &$q_m^P(\tilde{\lambda})$ \\
                                      & SO  & $q^C_s$ & $q^S_s$ & $q_s^P(\tilde{\lambda})$ \\
        \midrule
        Customer Net Benefit Rate & SO & $\text{U}^C$ & $\text{U}^S$ & $\text{U}^P(\tilde{\lambda})$ \\
        \midrule
        Revenue Rate& RM & $\text{Rev}^C$ & $\text{Rev}^S$ & $\text{Rev}^P$ \\
        \midrule
        Social Welfare Rate& SO & $\text{SW}^C$ & $\text{SW}^S$ & $\text{SW}^P$ \\
        \bottomrule
    \addlinespace[0.3em]
    \end{tabular}

    \label{Notations}
\end{table}


The main contributions of this paper are as follows:

\begin{itemize} 
    \item[1.] The impact of customer belief on equilibrium yields the relationship that $q^S_m\leq q^C_s = q^C_m = q^S_s\leq q_e^C = q^S_e$, under the condition that the expected waiting time is unbiased (see Section \ref{Chap: ImpactEq}).
    \item[2.] The impact of customer belief on revenue is mixed, depending on both the distribution of \(\Lambda\) and the system parameters. However, we identify specific conditions that determine the threshold at which the relationship shifts (see Section \ref{Chap: ImpactRev}). For a more intuitive understanding (see Figure \ref{RevCompare}).
    \item[3.] Consider optimizing over price and information disclosure strategy. From the SO’s perspective, eliminating information asymmetry always benefits social welfare, except when customers are generally pessimistic, with some individuals’ pessimism exceeding a specified threshold (see Section \ref{Chap: ImpactRevOpt}). From the RM’s perspective, revealing the true arrival rate is always a better strategy when customers are pessimistic, while calculations are needed to determine whether to reveal or conceal the information in other situations (see Section \ref{Chap: ImpactSW}).
    \item[4.] We provide managerial insights to guide managers in deciding which information to disclose, depending on customers’ beliefs and system parameters (see Section \ref{Sec:Insights})

\end{itemize}

\section{Joining Probability, Revenue and Social Welfare} \label{Sec: derivations}

In this section, we analyze the behavior of customers, the revenue maximizer, and the social optimizer, in each of the belief cases. 

\subsection{Classical Case}

In the classical case, all customers know the true arrival rate \(\lambda\) and thus are homogeneous in their strategies. They calculate the expected waiting time to be \(W(q\lambda)\), where all customers choose a joining probability \(q\) at equilibrium. When the service fee \(p\) is given, the expected net benefit of each customer is given by:

\begin{align}
    \text{U}^C(q) = R - p - CW(q\lambda).\label{u^u}
\end{align}

For the RM the revenue rate is given by $pq\lambda$. At equilibrium, we get from (\ref{u^u}) that the service fee is \(p = R - CW(q\lambda)\). For the SO, $p$ is considered to be a transfer payment. Thus, the benefit each customer yields upon completing the service is the customer's net benefit plus the transfer payment, which is also \(p =  R - CW(q\lambda) = \left[ R - CW(q\lambda)-p\right]+p\). Therefore, the revenue rate equals the net social welfare rate in the classical case, which can be expressed as:
\begin{align}
  \text{Rev}^C(q) = \text{SW}^C(q)= q\lambda \left(R - CW(q\lambda)\right). \label{Rev^Uq}
\end{align}
We prove the concavity of this function in Appendix \ref{Sec:Property}.

The goal of the SO is to optimize the long-run social welfare and the goal of the RM is to maximize the long-run revenue. The goals of the two types of manager coincide in this case, and their problem is given by:
\begin{align*}
    \underset{q \in [0,1]}{\max} \text{Rev}^C(q) = \underset{q \in [0,1]}{\max} \text{SW}^C(q).
\end{align*}
We denote the optimal joining probability for the SO as 
$q^C_s$ and the optimal joining probability for the RM 
as $q^C_m$. Since they solve the same optimization problem, we immediately have $q^C_m = q^C_s$. Furthermore, as shown in 
Edelson et al.\ \cite{edelson_congestion_1975}, $q^C_e\geq q^C_m = q^C_s$, where $q^C_e$ is the ``individual'' joining probability when
there is no system manager.  We discuss how information asymmetry impacts this relationship in Section 4.

From an alternative viewpoint, this problem can be reframed as a pricing problem. First, we establish the relationship between \(p\) and \(q\).
At equilibrium, the customers will settle on a strategy such that their expected net benefit is zero. Since each \( p \) induces an equilibrium joining probability \( q \) that solves \( \text{U}^C(q) = 0 \), we can express this relationship implicitly as
\begin{align} \label{q^C}
    q^C(p) = \min\left\{\frac{\xi(p)}{\lambda},\,1\right\},
\end{align}

where
\[
\xi(p) := \left(\frac{C s^2}{2(R - p - C/\mu)} + \frac{1}{\mu}\right)^{-1}.
\]

For some of our conclusions in this paper, in order to obtain explicit expressions, we assume exponential service times. In this case, the queue becomes an M/M/1 system, and $\xi(p)$ takes the simplified form  
\[
\xi(p) := \mu - \frac{C}{R - p}.
\]

In the rest of the paper, we continue to treat \( \xi \) as a function of price, but sometimes for notational convenience, we often write \( \xi \) without explicitly indicating its dependence on \( p \) when the context is clear.

Intuitively, $\xi$ represents the threshold effective arrival rate, the effective arrival rate at which a customer's utility is exactly zero. If we sample a customer belief $\tilde{\lambda}$ from the distribution $\Lambda$, the customer will choose joining probability $\tilde{q}$, so that $\xi = \tilde{q}\tilde{\lambda}$, which yields \(R - p - CW(\xi)\). Given fixed $p$, $\xi$ is a fixed constant as well. When the belief distribution $\Lambda$ shifts to the right(left), the distribution of joining probability $Q(p)$ shifts left(right) accordingly. In other words, the threshold effective arrival rate $\xi$ reflects the level of congestion that the customers tend to maintain to maximize their own benefit.

Finally, we can rewrite the expressions in terms of \(p\):

\begin{align} \label{RevUp}
  \text{Rev}^C(p) = \text{SW}^C(p) = 
  \begin{cases}
      p\,\xi(p) &  \lambda \geq \xi(p), \\
      p\lambda &  \lambda < \xi(p).
  \end{cases}
\end{align}

\subsection{Shared Belief Case}

Recall that in the shared belief case, customers all have a common belief regarding the arrival rate distribution. Then, customers compute the net expected benefit over the distribution of the arrival rate belief, which is given by:
\begin{align}
    \text{U}^S(q) = \mathbb{E}\left[ R-p-CW(q\Lambda)\right]. \label{u^s}
\end{align}

For the RM, the expected revenue rate is computed over the arrival rate belief distribution, given by:
\begin{align}
    \text{Rev}^S(q) = q\lambda \cdot \mathbb{E}\left[R-CW(q\Lambda)\right]\label{Rev^Sq}.
\end{align}
Note that the shared belief of the customers must be taken into account by the manager, as this belief distribution enters into their calculations once the entrance fee $p$ is posted. We prove the concavity of this function in Appendix \ref{Sec:Property}.

Nevertheless, the RM and SO no longer solve the same optimization problem in this scenario. The SO does not change to a different strategy from the classical case, because the SO cares about the actual net benefit (social welfare) obtained by the customer and optimizes the same problem as in the classical case based on the true arrival rate. Thus, we have:
\begin{align} \label{SWSq}
    \text{SW}^S(q) = \text{SW}^C(q)= q\lambda (R-CW(q\lambda)).
\end{align} 
However, the value of \(p\) required to make customers choose the same \(q\) may vary depending on the belief distribution, so \(\text{SW}^S(p)\) does not equal to \(\text{SW}^C(p)\). We prove the concavity of this function in Appendix \ref{Sec:Property}.

We denote the desired joining probability of the customers, the RM and the SO in this case by $q^S_e$, $q^S_m$ and $q^S_s$.

Similar to the classical case, let $q^S(p)$ denotes the equilibrium joining probability given the service fee \(p\). 
Unfortunately, this probability cannot generally be written in
a simple closed form. Denoting the solution to \(\text{U}^S(q)\) for a given \(p\) by \(\tilde{q}^S(p)\), we have

\begin{align} \label{q^S}
q^S(p) = \min\{\tilde{q}^S(p),1\}=  \begin{cases}
\tilde{q}^S(p) & \xi(p) \leq \bar{\lambda}, \\
1 & \xi(p) > \bar{\lambda}.
\end{cases}
\end{align}

The threshold point \(\bar{\lambda} = \mu - \frac{1}{\mathbb{E}[W(\Lambda)]}\) is given by the solution to \(\text{U}^S(1) = R-p-C\mathbb{E}[W(\Lambda)]\)= 0. This threshold coincides with the threshold \(\lambda\) in \eqref{q^C} if and only if \(\mathbb{E}[W(\Lambda)] = W(\lambda)\); that is, when all customers choose to join, the expected waiting time under the belief distribution matches the expected waiting time under the true arrival rate.

We can rewrite the revenue in terms of \(p\):
\begin{align}\label{RevSp}
  \text{Rev}^S(p) = p\lambda q^S(p)=
  \begin{cases}
p\lambda \tilde{q}^S(p) & \xi(p) \leq \bar{\lambda}, \\
p\lambda  & \xi(p) > \bar{\lambda}.
\end{cases}.
\end{align}
And the social welfare rate is given by

\begin{align} \label{SWSp}
\text{SW}^S(p) = 
\begin{cases}
\lambda \tilde{q}^S(p) \left[ R - CW(\tilde{q}^S(p)\lambda)\right] & \xi(p) \leq \bar{\lambda}, \\
\lambda\left[R-CW(\lambda)\right] & \xi(p) > \bar{\lambda}.
\end{cases}
\end{align}

\subsection{Private Belief Case}

For any customer with belief $\tilde{\lambda}$, the corresponding joining probability of the customer given \(p\) is \(\tilde{q} = \min\{\frac{\xi(p)}{\tilde{\lambda}},1\}\). In the private belief case, each customer derives their own joining strategy. We denote the joining probability of the individual with arrival rate belief \(\tilde{\lambda}\) under the desired system of the customer, the RM and the SO by $q^P_e(\tilde{\lambda})$, $q^P_m(\tilde{\lambda})$ and $q^P_s(\tilde{\lambda})$. 
 To characterize the overall probability of joining given price \(p\), we define the random variable:

\begin{align}\label{Qp}
Q(p): = \min\left\{ \frac{\xi(p)}{\Lambda}, 1\right\}.
\end{align}

Given a fixed \(p\), we take the expectation over the belief distribution \(\Lambda\) to get the revenue and social welfare in the private belief case:
\begin{align}\label{RevPp}
 \text{Rev}^P(p) = p\lambda \mathbb{E}\left[Q(p)\right] 
\end{align}
and
\begin{align} \label{SWPp}
\text{SW}^P(p) = \mathbb{E}\left[\lambda Q(p)\left(R-C\,W(\lambda Q(p))
\right)\right].   
\end{align}

\section{Impact of Information Asymmetry} \label{Sec: Impact}

In this section, we explore how the relationships among equilibrium, revenue, and social welfare differ across the three cases for the M/G/1 queue, from which we aim to gain insights into the manager's information disclosure and pricing strategies, as well as potential actions customers can take in response.

\subsection{Impact on Equilibrium} \label{Chap: ImpactEq}

In the private belief case, there is no uniform strategy for the joining probability. Therefore, in this section we focus on the comparison of equilibria for the unbiased and shared belief cases.

In the case of classical belief, it has been proved in \cite{edelson_congestion_1975} that $q_s^C = q_m^C \geq q_e^C$. In this section, we see how this conclusion changes in the shared belief case. We now present results on the relationships among the three equalibria, $q^S_e,q^S_m\text{ and } q^S_s$.

\begin{proposition} \label{prop:1}
    In an M/M/1 queue, $q^S_m\leq q^S_e$ for any $\Lambda$.
\end{proposition}
\begin{proof}
Let $\tilde{q}^S_m$ be the unique critical point of $\text{Rev}^S(q)$, then $q^S_m = \min\{\tilde{q}^S_m, 1\}$ is the equilibrium joining probability that maximize the revenue. Let $\tilde{q}^S_e$ be the unique root of $\text{U}^S(q)$ when \(p=0\), then $q^S_e = \min\{\tilde{q}^S_e, 1\}$ is the individual equilibrium joining probability. 

    First, assume $\tilde{q}^S_e = 0$, thus we have $R\mu = C$. Arguing by contradiction, suppose also that $\tilde{q}^S_m > 0$. Then it follows that
    \begin{align*}
       (\text{Rev}^S)'(\tilde{q}^S_m) = \lambda\, \mathbb{E}\left[R - R \cdot \left(\dfrac{\mu}{\mu - \tilde{q}^S_m\Lambda}\right)^2\right] < 0,
    \end{align*}
    which contradicts the fact that
 ${\text{Rev}^S}'(\tilde{q}^S_m) = 0$. Hence, $\tilde{q}^S_m = 0$ if $\tilde{q}^S_e = 0$.

    Now suppose instead that $\tilde{q}^S_e > 0$. Arguing by contradiction, suppose also that $\tilde{q}^S_m > \tilde{q}^S_e$. Recall that $\text{Rev}^S(0) = 0$ and $\text{Rev}^S(\tilde{q}^S_m) > \text{Rev}^S(\tilde{q}^S_e) = 0$, since $\tilde{q}^S_m$ is the unique critical point that maximizes $\text{Rev}^S(\cdot)$. Also, $\tilde{q}^S_e$ satisfies $\mathbb{E}\left[R - CW(\tilde{q}^S_e\Lambda)\right] = 0$ by definition, which implies that $\text{Rev}^S(\tilde{q}^S_e) = 0$. Collecting all of these observations yields
    \begin{align*}
        \dfrac{\tilde{q}^S_m - \tilde{q}^S_e}{\tilde{q}^S_m}\cdot\text{Rev}^S(0) +  \dfrac{\tilde{q}^S_e}{\tilde{q}^S_m} \cdot \text{Rev}^S(\tilde{q}^S_m) > 0 = \text{Rev}^S(\tilde{q}^S_e),
    \end{align*}
    which contradicts the strict concavity of $\text{Rev}^S(\cdot)$. Hence, $\tilde{q}^S_e > 0$ implies $\tilde{q}^S_m \leq \tilde{q}^S_e$,
i.e., $q^S_m \leq q^S_e$.
\end{proof}

We make the assumption that the manager cannot charge negative fee to the customers, which gives us the following result.

\begin{proposition} \label{prop:2}
In an M/M/1 queue, when the SO charges a non-negative fee to the customers, i.e., \(p^S_s \geq 0\), we have \(q^S_s \leq q^S_e\). Equality occurs when $p^S_s=0$. 
\end{proposition} 

\begin{proof}
Let $\tilde{q}^S_s$ be the unique critical point of $\text{SW}^S(q)$. Let $q^S_s := \min\{\tilde{q}^S_s, 1\}$, which is the joining probability that maximizes the social welfare in the shared belief case.

Since \( p = \mathbb{E}\left[ R - C W(q\Lambda) \right] \) is strictly decreasing in \( q \), \( q^S(p) \) is decreasing in \( p \). Thus, \( \tilde{q}^S_s = q^S(p^S_s) \leq q^S(0) =  \tilde{q}^S_e \), i.e., \(q^S_s \geq q_e^S\). When \( p^S_s = 0\) , \( q^S_s \) is equivalent to \( q^S_e \), corresponding to the case with no control.
\end{proof}

\begin{proposition}\label{prop:3}
    In an M/M/1 queue, for the shared belief case: (i) When the customers' belief about the distribution of the arrival rate satisfies $\mathbb{E}\left[ W^2(q_s^S\Lambda) \right] = W^2(q_s^S\lambda)$, where $\lambda$ is the true arrival rate and $W(\cdot)$ is the expected waiting time, we have that $q^S_m =  q^S_s$. (ii) Given that  $\mathbb{E}\left[ W(q_s^S\Lambda) \right] = W(q_s^S\lambda)$, we have $q^S_m \geq  q^S_s$. Equality occurs when $\Lambda=\lambda$ w.p.1.
\end{proposition} \label{2}

\begin{proof}
(i)  We have that  
\[
{\text{SW}^{S}}^{'}(q) = R\lambda - C \mu \lambda \cdot W^2(q\lambda)
\quad \text{and} \quad
{\text{SW}^{S}}^{''}(q) = -2C\lambda^2\mu W^3(q\lambda) \leq 0.
\]
Thus, \(\text{SW}^S(q)\) has a unique critical point \(\tilde{q}_s^S\).

Similarly, since  
\[
{\text{Rev}^{S}}^{'}(q) = R\lambda - C\mu\lambda \cdot \mathbb{E} \left[ W^2(q\Lambda) \right]
\quad \text{and} \quad
{\text{Rev}^{S}}^{''}(q) = \lambda \mathbb{E} \left[ -2C\mu\Lambda W^3(q\Lambda) \right] \leq 0,
\]
\(\text{Rev}^S(q)\) has a unique critical point \(\tilde{q}_m^S\).

The condition \(\mathbb{E}\left[ W^2(\tilde{q}_s^S \Lambda) \right] = W^2(\tilde{q}_s^S \lambda)\) implies that  
\[
(\text{Rev}^C)'(\tilde{q}_s^S) 
= R\lambda - C\mu \lambda \cdot \mathbb{E}\left[W^2(\tilde{q}_s^S \Lambda) \right] 
= R\lambda - C \mu \lambda \cdot W^2(\tilde{q}_s^S \lambda) 
= \text{SW}^{S'}(\tilde{q}_s^S) = 0.
\]
Considering the uniqueness of the critical point, we conclude that \(\tilde{q}_m^S = \tilde{q}_s^S\), i.e., \(q_m^S = q_s^S\).

(ii) Under the condition that $\mathbb{E} \left[ W(\tilde{q}_s^S\Lambda) \right] = W(\tilde{q}_s^S\lambda)$, we have
\begin{align*}
    (\text{Rev}^C)'(\tilde{q}_s^S) &  = R\lambda  -C\mu \lambda\cdot \mathbb{E}\left[W^2(\tilde{q}_s^S\Lambda) \right]\\
    &\leq R\lambda  -C\mu \lambda\cdot (\mathbb{E}\left[W^2(\tilde{q}_s^S\Lambda) \right]-\operatorname{Var}\left[W(\tilde{q}_s^S\Lambda) \right])\\ 
    &=  R\lambda  -C\mu \lambda\cdot \mathbb{E}^2\left[W(\tilde{q}_s^S\Lambda) \right]\\ 
    &= R\lambda - C \mu \lambda \cdot W^2(\tilde{q}_s^S \lambda)\\  
    &= \text{SW}^{S'}(\tilde{q}_s^S) = 0.
\end{align*}  
Recall that $\text{Rev}(\tilde{q}_m^S) = 0$ and that $\text{Rev}^S(\cdot)$ is strictly concave, the above inequality implies that $\tilde{q}_m^S\leq \tilde{q}_s^S$, i.e., $q_m^S\leq q_s^S$. 
\end{proof}

In this proposition, the first conclusion says that, even if the customers have heterogeneous beliefs, it is possible that the desired equilibrium of the RM and SO aligns. Notice that if $\operatorname{Var}[W(q_s^S\Lambda)] = 0$, then $\mathbb{E}[W(q_s^S\Lambda)] = W(q_s^S\lambda)$ indicates $\mathbb{E}[W^2(q_s^S\Lambda)] = W^2(q_s^S\lambda)$. Thus, the second conclusion in this proposition shows that, when the customer belief distribution produces an expected waiting time unbiased from that under the true arrival rate, the existence of variance in the estimation of expected waiting time causes the RM to prefer a higher joining probability than the SO.

Propositions \ref{prop:1} \ref{prop:2} and \ref{prop:3} together give us the following theorem about the relation of equilibria in the shared belief case.

\begin{theorem}
     In an M/M/1 queue, for the shared belief case, when $\mathbb{E}\left[ W(q_s^S\Lambda) \right] = W(q_s^S\lambda)$, we have $q^S_m \leq q^S_s \leq q^S_e$.
\end{theorem}

The following proposition bridges the relation between the equilibrium in the classical case and the shared belief case.
\begin{proposition}\label{prop:4} In an M/M/1 queue:
   (i) For any fixed $\lambda$, $q^S_s$ remains the same, independent of the distribution of $\Lambda$. Thus, $q^S_s = q^C_s$.
   (ii) Given that  $\mathbb{E}\left[ W(q_e^S\Lambda) \right] = W(q_e^S\lambda)$, \(q_e^S = q_e^C\).
\end{proposition}

\begin{proof}
    (i) This follows immediately from the formula for \(\text{SW}^S(q)\), which does not contains the random variable $\Lambda$. However, the corresponding fee \(p^C_s\) is not independent of \(\Lambda\). (ii) This is a direct consequence of comparing \eqref{u^u} and \eqref{u^s}.
\end{proof}

The first part of Proposition \ref{prop:4} can be intuitively explained by the fact that when a SO considers the fee charged to the customer as a transfer payment, the SO disregards it and focuses solely on finding the optimal joining probability, while neglecting the fee charged to customers to achieve this optimality, which is influenced by customer beliefs.

In conclusion, recall that in the classical case that $q^C_s = q^C_m \leq q^C_e$. Compare this with the shared belief case with the condition that $\mathbb{E}\left[ W(q_s^S\Lambda) \right] = W(q_s^S\lambda)$. Then, we have $q^S_m\leq q^C_s = q^C_m = q^S_s\leq q_e^C = q^S_e$.

\subsection{Impact on Revenue}\label{Chap: ImpactRev}
\subsubsection{Impact Under Fixed $p$ and $q$} \label{Chap: ImpactRevFixed}

For the M/G/1 queue, comparing the expressions for the revenue, in terms of \(p\), in \eqref{RevUp}, \eqref{RevSp}, and \eqref{RevPp} produces the following theorems that describe revenue relationships among the three information levels under a fixed service fee $p$. The relationship between $\text{Rev}^S$ and $\text{Rev}^C$ is primarily determined by the magnitude of \(\lambda\).
\begin{proposition} \label{thm:Rev_U_S}
In an M/G/1 queue: 
\begin{itemize}
\item[(i)] For a fixed \(p\), the following hold: 
\begin{itemize}
    \item[(a)] when \(\lambda \leq \mathbb{E}[\Lambda]\), \(\text{Rev}^S(p) \leq \text{Rev}^C(p),\)
    \item[(b)] when \(\mathbb{E}[\Lambda]<\lambda \leq \bar{\lambda}\), if  \(\xi(p)\geq\lambda\), then \(\text{Rev}^S(p) \leq \text{Rev}^C(p),\) 
    \item[(c)] when \(\lambda > \bar{\lambda}\), if \(\xi(p)\geq \bar{\lambda}\), then \(\text{Rev}^S(p) \geq \text{Rev}^C(p).\) 
    \end{itemize} 

    \item[(ii)] Given a fixed \(q\), when $\lambda > \mathbb{E}[\Lambda]$, there exists a $q_0 >0$ such that for any $q<q_0$,  we have $Rev^S(q)> Rev^C(q)$. When $\lambda < \mathbb{E}[\Lambda]$, there exists a  $q_0$ such that for any $q<q_0$,  we have $Rev^S(q)< Rev^C(q)$.
    
\end{itemize}
\end{proposition}

\begin{proof}
We first proof (a) and (b) together.  
\begin{itemize}
    \item[(*)] Suppose \(\xi(p) \leq\lambda \leq \mathbb{E}[\Lambda]\leq \bar{\lambda}\). Since \(\xi(p)\leq\lambda\) and \(\xi(p)\leq\bar{\lambda}\), it follows from \eqref{q^C} and \eqref{q^S} that
\[
\mathbb{E}[W(q^S(p)\Lambda)] = W(q^C(p)\lambda) = \frac{R-p}{C}.
\]
By the convexity of \(W(q\lambda)\), from Jensen's inequality:
\[
\mathbb{E}[W(q^S(p)\Lambda)] \geq W(q^S(p)\mathbb{E}[\Lambda]) \geq W(q^S(p)\lambda).
\]
Thus, 
\[W(q^S(p)\lambda)\leq W(q^C(p)\lambda).\]
By the monotonicity of \(W(q\lambda)\),  \(q^S(p)\leq q^C(p)\). Further, we have \(\text{Rev}^S(p)\leq \text{Rev}^C(p)\).
\item[(**)] When \(\lambda \leq \bar{\lambda} \leq \xi(p)\), \(q^S(p)= q^C(p)=1\), then \(\text{Rev}^S(p) = \text{Rev}^C(p) = p\lambda.\) 
\item[(***)]  When \(\lambda \leq \xi(p) \leq \bar{\lambda}\), \(q^S(p) \leq q^C(p)=1\), then \(\text{Rev}^S(p) \leq \text{Rev}^C(p) = p\lambda.\)

\end{itemize}

Part (*),(**) and (***) together indicate the area in (a) and (b) when \(\text{Rev}^S(p)\leq \text{Rev}^C(p)\).

(c) When \(\bar{\lambda} < \lambda \leq \xi(p)\), it follows from \eqref{q^C} and \eqref{q^S} that \(q^S(p)= q^C(p)=1\), then \(\text{Rev}^S(p) = \text{Rev}^C(p) = p\lambda.\) When \(\bar{\lambda} \leq \xi(p) \leq \lambda\), similarly, then \(\text{Rev}^C(p) \leq \text{Rev}^S(p) = p\lambda.\)

(ii) Let $F(q) = \text{Rev}^S(q)- \text{Rev}^C(q)$. Immediately, we have $F(0)=0$.

The first derivative is
\[
F'(q) = (\text{Rev}^C)'(q)  =\lambda C \left( 
\frac{q\lambda s^2}{2(1 - q\lambda/\mu)}
+ \frac{q\lambda s^2}{2(1 - q\lambda/\mu)^2}- \mathbb{E}\left[ 
\frac{q\Lambda s^2}{2(1 - q\Lambda/\mu)}
\right]-\mathbb{E}\left[\frac{q\Lambda s^2}{2(1 - q\Lambda/\mu)^2} 
\right]
\right).
\]
We have that $F'(0) = 0$.  
The second derivative is 
\[
F''(q) 
=  \lambda C \left(
\frac{\lambda s^2 (2 + q\lambda/\mu)}{2(1 - q\lambda/\mu)^2}
+ \frac{q\lambda^2 s^2 /\mu}{(1 - q\lambda/\mu)^3}-\mathbb{E}{ \left[
\frac{\Lambda s^2 (2 + q\Lambda/\mu)}{2(1 - q\Lambda/\mu)^2}
\right] - \mathbb{E}\left[\frac{q\Lambda^2 s^2 /\mu}{(1 - q\Lambda/\mu)^3}
\right]}
\right).
\]

When $\lambda > \mathbb{E}[\Lambda]$, we have $F''(0) = \lambda Cs^2 (\lambda - \mathbb{E}[\Lambda]) > 0.$ Expand $F(q)$ using Taylor expansion,

 $$F(q)=F(0)+F'(0)q+\dfrac{F''(0)}{2}q^2+o(q^2).$$
 Since $F''(0)>0$, for sufficiently small $q>0$, the leading term $\dfrac{F''(0)}{2}q^2$ dominates $o(q^2)$, ensuring that $F(q)>0$, i.e., $\text{Rev}^S(q)> \text{Rev}^C(q)$. 
 
 When $\lambda < \mathbb{E}[\Lambda]$, we have $F''(0) > 0.$. Similarly, we can show that for sufficiently small $q>0$, $\text{Rev}^S(q)< \text{Rev}^C(q)$.

\end{proof}

Next, we show that the relationship between $\text{Rev}^P(p)$ and $\text{Rev}^C(p)$ depends not only on the arrival rate $\lambda$ and the belief distribution $\Lambda$, but also on other system parameters $R$, $C$ and $\mu$. We further show that this dependency occurs only through $\xi$, rather than through any individual system parameter alone. We present the following theorem in the continuous arrival rate belief case; the discrete version is given in Appendix \ref{Sec:Appendix}.

\begin{theorem} \label{thm:Rev_U_P}
In an M/G/1 queue, suppose \(\Lambda\) follows a discrete distribution. Define the threshold function  \begin{align}
   M(\xi) &= 
   \begin{cases}
        1/\mathbb{E}(1/\Lambda)  & \text{if }\xi \leq  \lambda_{min} , \\
      \left( \int_{\lambda_{min}}^{\xi} \dfrac{f(\Lambda)}{\xi}d\Lambda +\int^{\lambda_{max}}_{\xi}\dfrac{1}{\Lambda} f(\Lambda)d\Lambda\right)^{-1} & \text{if } \lambda_{min}<\xi< \lambda_{max}.
   \end{cases}
   \end{align}
   Given $p$ fixed, 
   \begin{itemize}
       \item[(i)] When  $\xi\geq \lambda_{max}$,  \(\text{Rev}^P(p) = \text{Rev}^C(p)\).
        \item[(ii)] When  $\xi< \lambda_{max}$:
            \begin{itemize}
                \item[(a)] If $\lambda = M(\xi)$, then $\text{Rev}^P(p) = \text{Rev}^C(p)$; 
                \item[(b)]
                If $\lambda > M(\xi)$, then $\text{Rev}^P(p) > \text{Rev}^C(p)$; 
                \item[(c)]
                If $\lambda < M(\xi)$, then $\text{Rev}^P(p) < \text{Rev}^C(p)$.
                
            \end{itemize}
   \end{itemize}
   
   The threshold function $M(\xi)$ is continuous, monotonously increasing in terms of $\xi$ and piecewise convex.
\end{theorem}
\begin{proof}

(i) From \eqref{RevPp}, we have
$$\text{Rev}^P(p) = p\lambda \mathbb{E}[Q(p)],$$
where 
\begin{align} \label{EQp}
   \mathbb{E}[Q(p)] &= 
   \begin{cases}
        1   & \text{if }\xi\geq \lambda_{max} , \\
       \int_{\lambda_{min}}^{\xi} f(\Lambda)d\Lambda +\int^{\lambda_{max}}_{\xi}\xi\dfrac{1}{\Lambda} f(\Lambda)d\Lambda  & \text{if } \lambda_{min}<\xi< \lambda_{max} ,\\
      \xi \mathbb{E}[1/\Lambda] & \text{if } \xi \leq  \lambda_{min}.
   \end{cases}
\end{align}

When $\xi\geq \lambda$, the threshold comes from comparing $\text{Rev}^P(p) = p\lambda$ and $\text{Rev}^C(p) = p\lambda$, which means that for any \( \xi\), \(\text{Rev}^P(p) = \text{Rev}^C(p)\). 

(ii) \begin{itemize}
\item[(*)] When $\lambda \le \xi < \lambda_{\max}$, since $0<\mathbb{E}[Q(p)]<1$, we have
\[
\text{Rev}^C(p)=p\lambda \;>\; p\lambda\,\mathbb{E}[Q(p)] \;=\; \text{Rev}^P(p).
\]

\item[(**)] When $\xi<\lambda$, the equality $\text{Rev}^P(p)=\text{Rev}^C(p)$ holds iff $\lambda=M(\xi)$.

\quad(a) If $\xi \le \lambda_{\min}$, comparing $\text{Rev}^P(p)=p\lambda\,\xi\,\mathbb{E}[1/\Lambda]$ with $\text{Rev}^C(p)=p\xi$ gives equality iff \( \lambda=\frac{1}{\mathbb{E}[1/\Lambda]}=M(\xi).\)

\quad(b) If $\lambda_{\min}<\xi<\lambda$, define
\[
A(\xi):=\int_{\lambda_{\min}}^{\xi}\frac{f(\Lambda)}{\xi}\,d\Lambda
\;+\;\int_{\xi}^{\lambda_{\max}}\frac{1}{\Lambda}\,f(\Lambda)\,d\Lambda,
\qquad \lambda_{\min}<\xi<\lambda_{\max}.
\]
Then $\text{Rev}^P(p)=\text{Rev}^C(p)$ iff
\[
\lambda=\frac{1}{A(\xi)}=M(\xi).
\]
Moreover, if $\lambda>M(\xi)$ then $\text{Rev}^P(p)>\text{Rev}^C(p)$, and if $\lambda<M(\xi)$ then $\text{Rev}^P(p)<\text{Rev}^C(p)$.
\end{itemize}

Finally, since $M(\xi)=\xi/\mathbb{E}[Q(p)]$ with $0<\mathbb{E}[Q(p)]<1$, we have $M(\xi)>\xi$ whenever $\xi<\lambda_{\max}$. Combining Part (*) and (**) yields the claim in (ii).

Next, we show the properties of the threshold function \( M(\xi) \). When $\lambda_{min}<\xi< \lambda_{max}$, we have
\[
A^{'}(\xi) = -\int_{\lambda_{\min}}^{\xi} \frac{f(\Lambda)}{\xi^2}\,d\Lambda < 0,
\]
which implies
\[
M^{'}(\xi) = -\frac{A^{'}(\xi)}{A^2(\xi)} > 0.
\]
Furthermore,
\[
M^{''}(\xi) = -\frac{A^{''}(\xi)\,A(\xi) + 2\left[A^{'}(\xi)\right]^2}{A^3(\xi)},
\]
where
\[
A^{''}(\xi) = - \frac{f(\xi)}{\xi} - \int_{\lambda_{\min}}^{\xi} \frac{2f(\Lambda)}{\xi^3}\,d\Lambda.
\]
Then, we have
    $$M^{''}(\xi) = \dfrac{B(\xi)}{A^3(\xi)},$$
  where 

\begin{align*}
B(\xi) =\dfrac{2}{\xi^3} \int_{\lambda_{min}}^{\xi} f(\Lambda) d\Lambda 
\int_{\xi}^{\lambda_{max}} \dfrac{f(\Lambda)}{\Lambda} d\Lambda  > 0.
\end{align*}
Since \(M''(\xi)>0\) and \(M'(\xi)>0\), the threshold at interval $\lambda_{min}<\xi< \lambda_{max}$ is monotonously increasing and convex. Moreover, we can extend this conclusion to interval $\xi
\leq\lambda_{min}$ where \(M(\xi)\) is a constant.
      
We can verify that $M(\lambda_{min}) = 1/\mathbb{E}(1/\Lambda)$ and $M(\lambda_{max}) = \lambda_{max}$, showing that \(M(\xi)\) is continuous.
\end{proof}

Note that this theorem holds for any distribution $f(\Lambda)$. Given any distribution of $\Lambda$, $M(\xi)$ is directly computable and we can know the relationship between $\text{Rev}^P(p)$ and $\text{Rev}^C(p)$ for certain.

This theorem indicates that, compared to the classical belief case, the firm are less likely to benefit from lowering the price to incentivize joining behavior. For customers with a sufficiently low arrival rate belief, lowering the price will leads to a higher threshold effective arrival rate, but will not affect their strategy to join the queue with probability 1. The lower the $p$ is, the more percentage of customers become unaffected by the price change. If the fee $p$ is high enough to achieve a sufficiently low $\xi$, all the customers has a positive probability to balk. In this case, reducing the fee $p$ will incentivize all customers to join with a higher probability, as is in the classical belief case and in the shared belief case. 

\begin{proposition}\label{thm:Rev_S_P}
In an M/G/1 queue, given fixed \( p \) and the corresponding value \( \xi \), for any \( \lambda \):

\begin{itemize}
    \item[(i)] If \( \xi \geq \lambda_{\max} \), then \( \text{Rev}^P(p) = \text{Rev}^S(p) \).
    
    \item[(ii)] If \( \xi < \lambda_{\max} \), then there exists \( \xi_0 \in [\lambda_{\min}, \bar{\lambda}] \) such that:
    \begin{itemize}
        \item[(a)] If \( \xi = \xi_0 \), then \( \text{Rev}^P(p) = \text{Rev}^S(p) \); 
        \item[(b)] If \( 0 \leq \xi < \xi_0 \), then \( \text{Rev}^P(p) > \text{Rev}^S(p) \);
        
        \item[(c)] If \( \xi_0 < \xi < \lambda_{\max} \), then \( \text{Rev}^P(p) < \text{Rev}^S(p) \).
    \end{itemize}
\end{itemize}
\end{proposition}

\begin{proof}
(i) When $\lambda_{max} \leq \xi$, for any \(p\), $\mathbb{E}[Q(p)]=1$. Since \(\bar{\lambda}\leq \lambda_{max}\), we have \(q^S(p) = 1\), which indicates $\text{Rev}^P(p) = \text{Rev}^S(p)$.

(ii) When $\xi\leq \lambda_{min}$, since \(\bar{\lambda}\geq \lambda_{min}\geq \xi\), we have
\begin{align*}
  \text{Rev}^S(p) = p\cdot q^S(p) \lambda \leq p\cdot\xi(p)\cdot\dfrac{\lambda}{\mathbb{E}[\Lambda]}.  \label{Rev^S}
\end{align*}
We also have
\begin{align*}
  \text{Rev}^P(p) = p\cdot\xi(p)\cdot\mathbb{E}\left[\dfrac{\lambda}{\Lambda}  \right].
\end{align*}
Applying Jensen's inequality to the expressions above yields $\text{Rev}^S(p)\leq \text{Rev}^P(p)$.

We denote the inverse of $\xi(p)$ by $\xi^{-1}(x)$, which is given explicitly by \(\xi^{-1}(x) \;=\; R - \frac{C}{\mu} - \frac{C\,s^2\,\mu\,x}{2\,(\mu - x)},\)
and computes the price~$p$ required to achieve a given threshold effective arrival rate~$x$.  
Since \( \bar{\lambda}\leq\lambda_{max} \), it follows that
\[
\mathbb{E}[Q(\xi^{-1}(\bar{\lambda}))] < 1 = q^S(\xi^{-1}(\bar{\lambda})).
\]
Recall that
\[
\mathbb{E}[Q(\xi^{-1}(\lambda_{\min}))] \geq q^S(\xi^{-1}(\lambda_{\min})).
\]
By the monotonicity and continuity of both \( \mathbb{E}[Q(p)] \) and \( q^S(p) \), there exists \( \xi_0 \in [\lambda_{\min}, \bar{\lambda}]\) such that
\[
\mathbb{E}[Q(\xi^{-1}(\xi_0))] = q^S(\xi^{-1}(\xi_0)).
\]
Let \( p_0 = \xi^{-1}(\xi_0) \). Then \( \text{Rev}^P(p_0) = \text{Rev}^S(p_0) \), as desired.
\end{proof}

Note that this theorem indicates that the relationship between \(\text{Rev}^S(p)\) and \(\text{Rev}^P(p)\) mainly depends on what the threshold effective arrival rate is compared to the range of the customer belief,  but not the actual arrival rate $\lambda$.

To wrap up the conclusions in this section, Figure \ref{RevCompare} illustrates the relationship between the revenue for the three information levels by plotting the thresholds of the relationship switch between each pair of comparisons:

\begin{itemize}
    \item The purple dashed line corresponds to the first part of Proposition~\ref{thm:Rev_U_S}, 
which specifies the conditions used to determine the relationship between the classical 
and shared belief cases. However, the actual threshold cannot be expressed in an explicit analytical form.
The purple solid line therefore represents a possible threshold for the switch in the relationship. 
For any $\lambda$ below this line, we have $\text{Rev}^S(p) < \text{Rev}^C(p)$, 
while for any $\lambda$ above this line, we have $\text{Rev}^S(p) > \text{Rev}^C(p)$. 

    \item The red lines corresponds to Proposition \ref{thm:Rev_U_P},  which indicates that for any combination of \((\xi,\lambda)\) above this line, we have \(\text{Rev}^P(p) > \text{Rev}^C(p)\), and for any combination of \((\xi,\lambda)\) below this line, we have \(\text{Rev}^P(p) \leq \text{Rev}^C(p)\). The red curve indicates that the relationship between the private belief and the classical case depends on both the true arrival rate \(\lambda\) and other system parameters. An intuitive explanation is that the information increase from private to classical is the sum of information increase from private to shared (which depends on \(\xi\)) and from shared to classical (which depends on \(\lambda\)). 

    \item The blue line corresponds to the threshold indicated by Theorem~\ref{thm:Rev_S_P}, which shows that for any $\xi$ 
to the left of this line, we have $\text{Rev}^P(p) < \text{Rev}^S(p)$, 
and for any $\xi$ to the right of this line, we have $\text{Rev}^P(p) > \text{Rev}^S(p)$. Besides, the blue line falls within interval \([\lambda_{min},\bar{\lambda}]\) on the \(\xi\text{-axis}\). 
Note that the blue line intersects the point where the red line and the solid purple line cross; 
otherwise, the ordering of the comparisons would be contradictory.
\end{itemize}

\begin{figure}[h]
    \centering
    \includegraphics[width=1\textwidth]{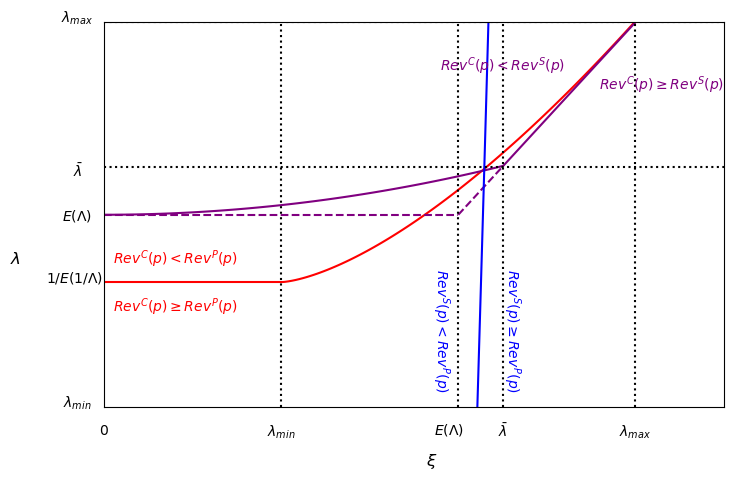}
    \caption{The thresholds for relation switch. }
    \label{RevCompare}
\end{figure}

From this figure, we observe that the revenue comparison between the shared-belief and private-belief cases is more sensitive to $\xi$, whereas the comparison between the classical and shared-belief cases is more sensitive to $\lambda$. Moreover, the three thresholds partition the $(\xi,\lambda)$ space into three regions, within each of which a single case dominates in revenue.

\subsubsection{Optimizing Revenue Over $p$} \label{Chap: ImpactRevOpt}
Thus far, we have examined how information asymmetry affects revenue when either \(p\) or \(q\) is held fixed. However, we are ultimately interested in comparing the optimal revenue across the three scenarios, as this comparison provides insight into decisions about the appropriate level of information disclosure. Denote the optimal revenue by $\text{Rev}^C_{opt}$, $\text{Rev}^{S}_{opt}$ and $\text{Rev}^{P}_{opt}$. 

For the M/M/1 queue, Theorem \ref{thm:Rev_opt_P_U} states that either the private belief or classical case yields the maximum possible benefits depends on which side of the red curve in Figure \ref{RevCompare} the pair (\(\xi,\lambda\)) falls to.

\begin{theorem} \label{thm:Rev_opt_P_U}
In an M/M/1 queue:
\begin{itemize}
    \item[(i)] When $\lambda\leq 1/\mathbb{E}(1/\Lambda)$, $\text{Rev}^{P}_{opt}\leq\text{Rev}^C_{opt}$.
    \item[(ii)] When $\lambda > 1/\mathbb{E}(1/\Lambda)$, let $\xi^C = \mu - \sqrt{\frac{C\mu}{R}}$,  then
    \begin{itemize}
        \item[(a)] If $\lambda \geq M(\xi^C)$,  $\text{Rev}^C_{opt}\leq\text{Rev}^{P}_{opt}$.  
        \item[(b)] If $\lambda < M(\xi^C)$, $\text{Rev}^{P}_{opt}<\text{Rev}^C_{opt}$.
    \end{itemize}
        
\end{itemize}
\end{theorem}

\begin{proof}
    (i) This is a direct consequence of Theorem \ref{thm:Rev_U_P}.

    (ii) (a) We have
    $$(\text{Rev}^{C})'(\xi) = \begin{cases}
        R - \frac{C\mu}{(\mu-\xi)^2} & \lambda\geq \xi, \\
        -\frac{C\lambda}{(\mu - \xi)^2} & \lambda < \xi.
        
    \end{cases} $$
    Let $(\text{Rev}^C)'(\xi) = 0$, we get that the maximizer of $\text{Rev}^C(\xi)$ is $\xi^C = \min {\left\{\mu - \sqrt{\frac{C\mu}{R}},\lambda\right\}}$.

    Since \(\xi^C\leq \lambda\leq \lambda_{max}\), according to Theorem \ref{thm:Rev_U_P}, if $M^{-1}(\lambda) \geq \xi^C$, i.e., $\lambda \geq M(\xi^C)$, we have $\text{Rev}^P(\xi(p))\geq \text{Rev}^C(\xi(p))$ for any $\xi$. Thus, $$\max_{\xi}\text{Rev}^P(\xi) \geq \text{Rev}^P(\xi^C)\geq \text{Rev}^C(\xi^C) = \max_{\xi}\text{Rev}^C(\xi).$$
   This is equivalent to  $\text{Rev}^P_{opt} \geq \text{Rev}^C_{opt}$.
   
(b)     Denote the maximizer of $\text{Rev}^P(\xi)$ with $\xi^{P}$, then
$$[\text{Rev}^P(\xi)] = \begin{cases}
    \lambda \mathbb{E}[1/\Lambda]\left[R - \frac{C\mu}{(\mu - \xi)^2}\right] & \xi\leq \lambda_{min},\\
    (2R - \frac{C}{\mu -\xi}- \frac{C\mu}{(\mu-\xi)^2})\lambda\int_{\lambda_{min}}^{\xi} \dfrac{f(\Lambda)}{\xi}d\Lambda+\lambda \int^{\lambda_{max}}_{\xi} \dfrac{f(\Lambda)}{\Lambda}d\Lambda & \lambda_{min}<\xi \leq \lambda_{max},\\
    -\frac{C\lambda}{(\mu - \xi)^2}& \xi >\lambda_{max}.\\
\end{cases}$$

Now we want to show that $\xi^P\geq\xi^C$. If \(\xi^C \leq \lambda_{\min}\), then 
\[
\xi^C = \mu - \sqrt{\tfrac{C\mu}{R}}, 
\quad \text{and} \quad 
(\text{Rev}^P)'(\xi^C) = 0.
\]
If instead \(\xi^C > \lambda_{\min}\), consider two cases. First, when \(\xi^C = \mu - \sqrt{\tfrac{C\mu}{R}}\), that is, when 
\(\mu - \sqrt{\tfrac{C\mu}{R}} \leq \lambda\), we obtain
\[
(\text{Rev}^P)'(\xi^C) = \Bigl(R - \sqrt{\tfrac{CR}{\mu}}\Bigr)\lambda 
\int_{\lambda_{\min}}^{\xi^C} \frac{f(\Lambda)}{\xi^C}\, d\Lambda
+ \lambda \int_{\xi^C}^{\lambda_{\max}} \frac{f(\Lambda)}{\Lambda}\, d\Lambda.
\]
Since \(R \geq \tfrac{C}{\mu}\), it follows that 
\((\text{Rev}^P)'(\xi^C) > 0\), which is also true when \(\Lambda\) follows a discrete distribution.
Second, when \(\xi^C = \lambda\), which occurs if 
\(\lambda < \mu - \sqrt{\tfrac{C\mu}{R}}\), we have
\[
2R - \frac{C}{\mu - \lambda} - \frac{C\mu}{(\mu - \lambda)^2} 
\;\;\geq\;\; R - \sqrt{\tfrac{CR}{\mu}} 
\;\;\geq\; 0.
\]
This again implies that \([\text{Rev}^P(\xi^C)]' > 0\).
Therefore, in all cases it holds that
\[
(\text{Rev}^P)'(\xi^C) \geq 0,
\]
which indicates that \(\xi^P \geq \xi^C\).


     Since $\lambda < M(\xi^C)< M(\xi^P)$, we have $\text{Rev}^P(\xi^P)<\text{Rev}^C(\xi^P)$  according to Theorem \ref{thm:Rev_U_P}. Thus,
     $$\max_{\xi}\text{Rev}^C(\xi) > \text{Rev}^C(\xi^P)> \text{Rev}^P(\xi^P) = \max_{\xi}\text{Rev}^P(\xi).$$
      This is equivalent to $\text{Rev}^C_{opt} > \text{Rev}^P_{opt}$.
\end{proof}

For the other two pairs of comparison, it is difficult to derive a clean expression to distinguish the dominant information level. Proposition \ref{thm:Rev_U_S} tells us that when \(\lambda \leq \mathbb{E}[\Lambda]\), we have that $\text{Rev}^C_{opt} \geq \text{Rev}^S_{opt}$. As a result, the optimal revenue is always higher in the classical case than in the shared belief and the private belief case. However, when \(\lambda > \mathbb{E}[\Lambda]\), we are not sure about the relationship.

\subsection{Impact on Social Welfare} \label{Chap: ImpactSW}

The expression for social welfare in the private belief case \eqref{SWPp} has a complex structure, making it difficult to compare under a fixed \(p\). However, in the social optimizer's problem, it is easier to directly compare the optimal social welfare across different information levels.

\begin{theorem} \label{SW_opt}
In an M/M/1 queue, The optimal social welfare for the three classes satisfies
\[
\text{SW}_{opt}^P \leq \text{SW}_{opt}^C = \text{SW}_{opt}^S,
\]
in either of the following cases:
\begin{enumerate}
    \item[(i)] $\xi(0) \geq \bar{\lambda}$.
    \item[(ii)] $\xi(0) <\bar{\lambda}$ and the belief distribution $\Lambda$ has support $[\lambda_{\min}, \lambda_{\max}]$ satisfying $\lambda_{\max} \leq (1+\epsilon)\lambda$, where \(\epsilon  = \sqrt{\frac{C}{R\mu}}\).
\end{enumerate}

\end{theorem}
\begin{proof}
(i) From \eqref{SWSq}, we have $\text{SW}^S_{opt} = \text{SW}^C_{opt}$. Now we want to show that $\text{SW}_{opt}^P \leq \text{SW}_{opt}^S$. 

Apply Jensen's inequality to \eqref{SWPp}. Since $\text{SW}^P(p)$ is concave,
\[
\text{SW}^P(p) \leq \lambda \mathbb{E}[Q(p)] \left[ R - \frac{C}{\mu - \lambda \mathbb{E}[Q(p)]} \right].
\]
Recall that the social welfare when customers have shared beliefs is
\[
\text{SW}^S(p) = \lambda q^S(p) \left[ R - \frac{C}{\mu - \lambda q^S(p)} \right].
\]

Next, we compare the ranges of support in the two cases. The maximum equilibrium joining probability in the shared belief case is achieved when no fee is charged:
\begin{align*}
    \max_p q^S(p) = q^S(0) = \begin{cases}
\tilde{q}^S(0) & \xi(0) \leq \bar{\lambda}, \\
1 & \xi(0) > \bar{\lambda}.
\end{cases}
\end{align*}
Similarly, in the private belief case, the maximum equilibrium joining probability is given by:
\begin{align*}
  \max_p{\mathbb{E}[Q(p)]} = \mathbb{E}[Q(0)] = \begin{cases}
        1   & \text{if }\xi(0)\geq \lambda_{max} , \\
       \int_{\lambda_{min}}^{\xi(0)} f(\Lambda)d\Lambda +\int^{\lambda_{max}}_{\xi(0)}\xi(0)\dfrac{1}{\Lambda} f(\Lambda)d\Lambda  & \text{if } \lambda_{min}<\xi(0)< \lambda_{max} ,\\
      \xi(0) \mathbb{E}[1/\Lambda] & \text{if } \xi(0) \leq  \lambda_{min}.
   \end{cases} 
\end{align*}

If $\xi(0)\geq \bar{\lambda}$, then $q^S(p)\in[0,1]$ and $\mathbb{E}[Q(p)]\in[0,\mathbb{E}[Q(0)]]$ for all $p$, which is also true when \(\Lambda\) follows a discrete distribution. Moreover, the support of $\mathbb{E}[Q(p)]$ is contained in the support of $q^S(p)$. Thus, the maximum of $\text{SW}^S(p)$ is no less than the upper bound of $\text{SW}^P(p)$, i.e.,
\[
\max_{q^S(p)\in [0,1]} \lambda q^S(p)\left[ R - \frac{C}{\mu - \lambda q^S(p)} \right] 
\;\geq\; 
\max_{\tilde{q}(p)\in [0,\mathbb{E}[Q(0)]] } \lambda \mathbb{E}[Q(p)] \left[ R - \frac{C}{\mu - \lambda \mathbb{E}[Q(p)] } \right].
\]
It follows that
\[
\text{SW}^P_{opt} \leq \text{SW}^S_{opt  }.
\]

(ii) If \(\xi(0)< \bar{\lambda}\), the maximum of \(q^S(p)\) satisfies the following equation:
\begin{align*}
    q^S(0) = R-C\mathbb{E}\left[\frac{1}{\mu-q^S(0)\Lambda}\right] = 0.
\end{align*}
Apply Jensen Inequality, we have
\begin{align*}
    \frac{R}{C} = \mathbb{E}\left[\frac{1}{\mu-q^S(0)\Lambda}\right]\geq \frac{\mathbb{E}\left[\frac{1}{\Lambda}\right]}{\mu\mathbb{E}\left[\frac{1}{\Lambda}\right]-q^S(0)},
\end{align*}
which implies that 
\[q^S(0) \leq \xi(0) \mathbb{E}[1/\Lambda]\leq \mathbb{E}[Q(0)].\]
Thus, the support of $q^S(p)$ is contained in the support of $\mathbb{E}[Q(p)]$.

Recall that the optimal of the upper bound of $\text{SW}^P(p)$ is achieved when \(\mathbb{E}[Q(p)] = (\mu - \sqrt{\frac{C\mu}{R}})\frac{1}{\lambda}:= q^*\). We want to show that \(q^*\) is contained within \([0,q^S(0)]\), such that 
\[
\max_{q^S(p)\in [0,\tilde{q}^S(0)]} \lambda q^s(p)\left[ R - \frac{C}{\mu - \lambda q^S(p)} \right] = \max_{\mathbb{E}[Q(p)]\in [0,\mathbb{E}[Q(0)]] } \lambda \tilde{q}(p) \left[ R - \frac{C}{\mu - \lambda \mathbb{E}[Q(p)] } \right].
\]
Hence, \[
\text{SW}^P_{opt} \leq \text{SW}^S_{opt  }.
\]

When \(p = 0\), we have
\[\text{U}^S(q^*) = R - \mathbb{E}\left[\frac{C}{\mu - q^*\Lambda}\right]\geq R - \frac{C}{\mu - q^*\lambda_{max}} \geq R  - \frac{C}{\mu - \left(\mu - \sqrt{\frac{C\mu}{R}}\right)(1+\epsilon)} = 0, \] 
where the last inequality is due to the condition that \(\lambda_{max}\leq (1+\epsilon)\lambda\).
Since \(U^S(q^S(0)) = 0 \leq  \text{U}^S(q^*) \) and that \(\text{U}^S(q)\) is monotonously decreasing in terms of \(q\), we have \(q^*\leq q^S(0)\), which completes the proof.

\end{proof}

The optimal social welfare in the classical and shared belief cases is always the same. 
The condition $\xi(0) \geq \bar{\lambda}$ is equivalent to $\mathbb{E}[W(\Lambda)] \leq \tfrac{R}{C}$, 
indicating that customers expect a low arrival rate, i.e., optimistic about the waiting time. In this setting, the optimal social welfare under private beliefs is strictly lower 
than that under classical and shared beliefs. 
The same conclusion holds when $\xi(0) < \bar{\lambda}$, i.e., when $\mathbb{E}[W(\Lambda)] > \tfrac{R}{C}$ and customers are pessimistic, 
provided that the spread of beliefs on the right-hand side is sufficiently limited.


\section{Managerial Insight on Information Revelation} \label{Sec: Insights}

In this section, we discuss the strategies managers can take to maximize revenue or optimize social welfare by strategically revealing information and pricing the service appropriately.
The amount of information increases in the order of private belief, shared belief, and classical belief case. The manager can choose to release the customer belief distribution or the true arrival rate, however, this process is not reversible. We propose the general steps managers can follow in an M/M/1 queue to evaluate whether revealing information is beneficial.

\textbf{Strategy for RM}: It involves more computation for the RM to find the optimal strategy. The RM's best strategy depends on the belief distribution, the true arrival rate and the system parameters.

\begin{itemize}
    \item \(\lambda\leq 1/\mathbb{E}[1/\Lambda]\): This is when customers are  \textbf{pessimistic} about the arrival rate. The RM should reveal the true arrival rate and then optimize over \(p\).
    \item \( 1/\mathbb{E}[1/\Lambda]<\lambda\leq\mathbb{E}[\Lambda]\): This is when customers are \textbf{neutral} about the arrival rate. The RM should compute for \(M(\xi^C)\) and compare it with \(\lambda\). If \(\lambda < M(\xi^C)\), then release the true arrival rate; if not, do not release any information.
    \item \( \lambda>\mathbb{E}[\Lambda]\): This is when customers are \textbf{optimistic} about the arrival rate. In general, it is better for the RM to keep the customer belief private when planning to charge a high price and release some information when planning to charge a low price.
    
\end{itemize}

\textbf{Strategy for SO}: The strategy for the SO is simpler and more straightforward. To achieve optimal social welfare, the SO should release the customers' belief distribution or the true arrival rate when the customers are optimistic. In cases where the customers are pessimistic, as long as all customers expect an arrival rate no higher than a certain range, the manager should still reveal the information; otherwise, the manager should be cautious about releasing it.

\section{Numerical Results}\label{Sec: Numerics}

In this section, we perform a series of numerical experiments to verify and give a more intuitive explanation to our analysis. 

\subsection{Illustration 1: Equilibria}

We verify the results in section 4.1 with customer belief distribution following a two-point discrete distribution with $\mathbf{P(\Lambda=2.2)=0.5}$ and $\mathbf{P(\Lambda=3.8)=0.5}$. The system parameters are set to \( \mathbf{R} = \mathbf{5} \), \( \mathbf{C} = \mathbf{4} \), \( \boldsymbol{\mu} = \mathbf{4} \) and \( \boldsymbol{\lambda} = \mathbf{3} \). In Figure \ref{equilibria}, the x-axis of the intersection point of the curves with the \(R/C\) line indicates the root of the derivative of equation \eqref{Rev^Sq},\eqref{Rev^Uq}, \eqref{u^s} and \eqref{u^u} from left to right correspondingly. 

\begin{figure}[H]
    \centering
    \includegraphics[width=0.8\textwidth]{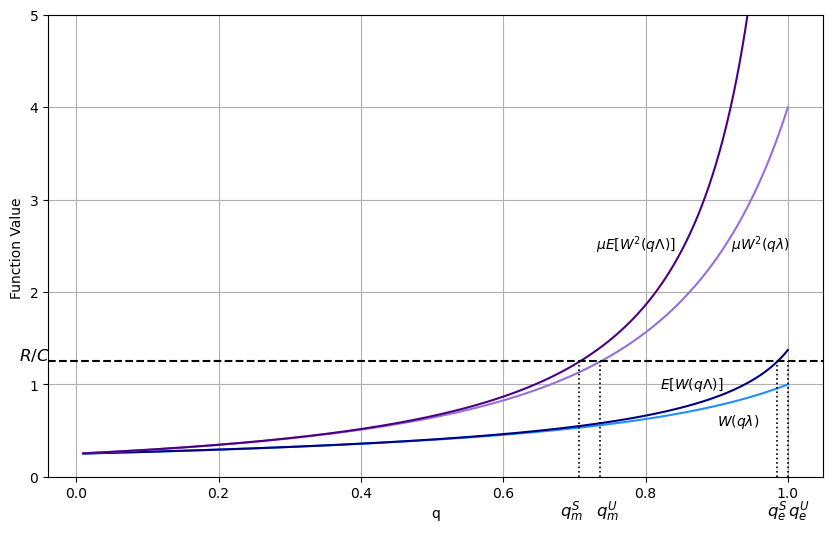}
    \caption{Comparison of desirable equilibrium for customer, RM and SO. The x-axis of the intersection point of the curves with the \(R/C\) line indicates the corresponding equilibria. }
    \label{equilibria}
\end{figure}

The order of these intersection points shows that \(q_m^S < q_m^C = q_s^C = q_s^S< q_e^S < q_e^C = 1\), which aligns with our analysis.

\subsection{Illustration 2: Revenue}
We present a few experiments to show how the mean and variance of belief distribution impact revenue and interpret the results with the threshold  Figure \ref{RevCompare}. In this subsection, we fix system parameter to  \( \mathbf{R} = \mathbf{5} \), \( \mathbf{C} = \mathbf{5} \), \( \boldsymbol{\mu} = \mathbf{5} \) and \( \boldsymbol{\lambda} = \mathbf{4.2} \) and the customer belief distribution \(\Lambda\) follows the uniform distribution.

\textbf{Experiment 1}
Keep the service fee \( p \) and the variance of the belief distribution \( \Lambda \) fixed. By gradually changing the mean of the uniform distribution, we obtain the results shown in Table \ref{num:mean}. Columns 2-4 indicate the potential threshold values at which a switch in the relationship could occur, corresponding to each belief distribution.

\begin{table}[h]
    \centering
    \begin{tabular}{c|c|c c c|c c c}
        \toprule
        & $\Lambda$ & $1/\mathbb{E}[1/\Lambda]$ & $\mathbb{E}[\Lambda]$ &  $M(\xi)$ & $\text{Rev}^P$& $\text{Rev}^S$ &  $\text{Rev}^C$\\

        \hline
1  & \text{U}(3.4, 4.0) & 3.692 & 3.7 & 3.719 & 6.05 & 6.048 & 5.357 \\
2  & \text{U}(3.5, 4.1) & 3.792 & 3.8 & 3.797 & 5.926 & 5.891 & 5.357 \\
3  & \text{U}(3.6, 4.2) & 3.892 & 3.9 & 3.892 & 5.781 & 5.741 & 5.357 \\
4  & \text{U}(3.7, 4.3) & 3.992 & 4.0 & 3.992 & 5.636 & 5.599 & 5.357 \\
5  & \text{U}(3.8, 4.4) & 4.093 & 4.1 & 4.093 & 5.498 & \cellcolor{violet!40}5.464 & 5.357 \\
6  & \text{U}(3.9, 4.5) & 4.193 & 4.2 & 4.193 &\cellcolor{red!40}5.366 & \cellcolor{violet!40}5.335 & 5.357 \\
7  & \text{U}(4.0, 4.6) & 4.293 & 4.3 & 4.293 &\cellcolor{red!40}5.241 & 5.212 & 5.357 \\
8  & \text{U}(4.1, 4.7) & 4.393 & 4.4 & 4.393 & 5.122 & 5.094 & 5.357 \\
9  & \text{U}(4.2, 4.8) & 4.493 & 4.5 & 4.493 & 5.007 & 4.982 & 5.357 \\
10 & \text{U}(4.3, 4.9) & 4.593 & 4.6 & 4.593 & 4.898 & 4.874 & 5.357 \\
11 & \text{U}(4.4, 5.0) & 4.694 & 4.7 & 4.694 & 4.794 & 4.771 & 5.357 \\

        \bottomrule
    \end{tabular}
    \caption{Revenues under changing mean of belief distribution \(\Lambda\) when \(p = 1.5\).}
    \label{num:mean}
\end{table}

While the revenue in the classical case remains the same, revenue in both the private and shared belief case shows a decreasing trend when the belief distribution shifts to the right. In this parameter setting, the threshold arrival rate is \(\xi = 3.571\). Comparing the private belief case and the classical case, we notice that the switch in the relationship (highlighted in red) happens when \(M(\xi)\) goes from 4.193 to 4.293, where there is a switch from  \(M(\xi)<\lambda\) to \(M(\xi)>\lambda\). This observation demonstrates the threshold indicated in Theorem \ref{thm:Rev_U_P}. Comparing the shared belief case with the classical case, we see a switch (highlighted in purple) when the arrival rate belief changes from $\Lambda \sim U(3.8, 4.4)$ to $\Lambda \sim U(3.9, 4.5)$. Correspondingly, the expectation shifts from $\mathbb{E}[\Lambda] = 4.2 = \lambda$ to $\mathbb{E}[\Lambda] = 4.1 < \lambda$, while it remains that $\bar{\lambda}  > \xi$. This result aligns with Theorem \ref{thm:Rev_U_S}, however, the theorem itself only guarantees that the switch will not happen for distribution with expected value higher than \(\lambda\), and happen indefinitely for expected value lower than \(\lambda\).

\textbf{Experiment 2} The spread of the uniform belief distribution reflects how variant are the customers' belief about the arrival rate. We control the mean of the belief distribution at three different levels and fix service fee \(
p\), gradually changing the spread of the distribution at each level. 

\begin{table}[H]
    \centering
    \begin{tabular}{c|c|c| c |c c c}
        \toprule
       & Level & $\Lambda$  & $\mathbb{E}[\Lambda]$ &  $\text{Rev}^P$ &  $\text{Rev}^S$ &  $\text{Rev}^C$\\

        \hline
        1  & \multirow{5}{*}{Unbiased}    & 4.2 & 4.2 & 5.357 & 5.357 & 5.357 \\
        2  &                              & \( \text{U}(4.1, 4.3) \) & 4.2 & 5.358 & 5.355 & 5.357 \\
        3  &                              & \( \text{U}(4.0, 4.4) \) & 4.2 & 5.361 & 5.347 & 5.357 \\
        4  &                              & \( \text{U}(3.9, 4.5) \) & 4.2 & 5.366 & 5.335 & 5.357 \\
        5  &                              & \( \text{U}(3.8, 4.6) \) & 4.2 & 5.373 & 5.317 & 5.357 \\
        \hline
        6  & \multirow{5}{*}{Optimistic}  & 3.8 & 3.8 & 5.921 & 5.921 &5.357\\
        7  &                              & \( \text{U}(3.7, 3.9) \)  & 3.8 & 5.922 & 5.918 & 5.357 \\
        8  &                              & \( \text{U}(3.6, 4.0) \)  & 3.8 & 5.927 & 5.907 & 5.357 \\
        9  &                              & \( \text{U}(3.5, 4.1) \) & 3.8 & 5.926 & 5.891 & 5.357 \\
        10  &                              & \( \text{U}(3.4, 4.2) \)  & 3.8 & 5.910 & 5.868 & 5.357 \\
        \hline
        11  & \multirow{5}{*}{Pessimistic} & 4.6 & 4.6 & 4.891 & 4.891 &5.357\\
        12  &                              &\( \text{U}(4.5, 4.7) \) & 4.6 & 4.892 & 4.889 & 5.357 \\
        13 &                              & \( \text{U}(4.4, 4.8) \)  & 4.6 & 4.894 & 4.884 & 5.357 \\
        14 &                              & \( \text{U}(4.3, 4.9) \) & 4.6 & 4.898 & 4.874 & 5.357 \\
        15 &                              & \( \text{U}(4.2, 5.0) \)  & 4.6 & 4.904 & 4.861 & 5.357 \\
        \bottomrule
    \end{tabular}
    \caption{Revenues under changing spread of belief distribution \(\Lambda\) when \(p = 1.5\) and the distribution is unbiased, optimistic or pessimistic.}
    \label{num:spread}
\end{table}
For the shared belief case, the results in Table \ref{num:spread} show that increasing spread harms revenue, as customers expect longer waiting times when they learn the varied belief. 

For the private belief case, increasing spread boosts revenue when customers are unbiased or pessimistic. However, when customers are optimistic, revenue starts decreasing when the belief is too spread out. The reason behind this is that the threshold joining probability when \( p = 1.5 \) is \( \xi = 3.571 \). More customers who believe in a lower arrival rate than this will not increase their joining probability, while more customers who believe in a higher arrival rate will always lower their joining probability.

\textbf{Experiment 3} In this experiment, we show how setting different service fee \(p\) affects revenue when the customer belief distribution is given. We also demonstrate how optimistic and pessimistic customers can affect optimal strategy.

\begin{table}[H]
    \centering
    \begin{minipage}{0.45\textwidth} 
        \centering
        \resizebox{\textwidth}{!}{
            \begin{tabular}{c|c|c|c|c|c}
                \toprule
                & $p$ & $\xi$ & $\text{Rev}^P$ & $\text{Rev}^S$ &  $\text{Rev}^C$ \\
                \hline
                1 & 0.1  & 3.98   & 0.420  & 0.42  & 0.398 \\
                2 & 0.5  & 3.889  & 2.092  & 2.1   & 1.944 \\
                3 & 0.9  & 3.78   & \cellcolor{blue!50}3.722 & \cellcolor{blue!50}3.75  & 3.402 \\
                4 & 1.3  & 3.649  & \cellcolor{blue!50}5.243 & \cellcolor{blue!50}5.23  & 4.743 \\
                5 & 1.7  & 3.485  & 6.554  & 6.534 & 5.924 \\
                6 & 2.1  & 3.276  & 7.61   & 7.59  & 6.879 \\
                7 & 2.5  & 3.0    & 8.297  & 8.278 & 7.5   \\
                8 & 2.9  & 2.619  & 8.403 & 8.386 & 7.595 \\
                9 & 3.3  & 2.059  & 7.516  & 7.504 & 6.794 \\
                10 & 3.7  & 1.154  & 4.723  & 4.717 & 4.269 \\
                \bottomrule
            \end{tabular}
        }
        \caption{Revenues under changing \(p\) when customers are optimistic with belief distribution $\Lambda\sim$U(3.6,4.0).}
        \label{num:p1}
    \end{minipage}%
    \hspace{0.05\textwidth} 
    \begin{minipage}{0.45\textwidth} 
        \centering
        \resizebox{\textwidth}{!}{
            \begin{tabular}{c|c|c|c|c|c}
                \toprule
                & $p$ & $\xi$ & $\text{Rev}^P$ & $\text{Rev}^S$ &  $\text{Rev}^C$ \\
                \hline
        1  & 0.1  & 3.98  & 0.364  & 0.362  & 0.398  \\
        2  & 0.5  & 3.889  & 1.776  & 1.771  & 1.944  \\
        3  & 0.9  & 3.78  & 3.109  & 3.101  & 3.402  \\
        4  & 1.3  & 3.649  & 4.334  & 4.323  & 4.743  \\
        5  & 1.7  & 3.485  & 5.413  & 5.401  & 5.924  \\
        6  & 2.1  & 3.276 & 6.285  & 6.274  & 6.879  \\
        7  & 2.5  & 3.0    & 6.852  & 6.841  & 7.5    \\
        8  & 2.9  & 2.619  & 6.939  & 6.930  & 7.595  \\
        9  & 3.3  & 2.059  & 6.207  & 6.201  & 6.794  \\
        10 & 3.7  & 1.154  & 3.900  & 3.897  & 4.269  \\
                \bottomrule
            \end{tabular}
        }
        \caption{Revenues under changing \(p\) when customers are pessimistic with belief distribution $\Lambda\sim$U(4.4,4.8).}
        \label{num:p2}
    \end{minipage}
\end{table}
Comparing the private belief case and the shared belief case in Table \ref{num:p1}, we observe that the switch in the relationship (highlighted in blue) occurs when \(\xi\) shifts from 3.78 to 3.649, aligning with Theorem \ref{thm:Rev_S_P}. In this experiment, with $\lambda_{\min} = 3.6$ and $\bar{\lambda} = 3.833$, we can also deduce that the switching point $\xi_0 \in [\lambda_{\min}, \bar{\lambda}]$.

For optimistic customers in Table \ref{num:p1}, since \(1/\mathbb{E}[1/\Lambda] = 3.796 < \lambda\), Theorem \ref{thm:Rev_opt_P_U} indicates that \(\text{Rev}^P_{opt} > \text{Rev}^C_{opt}\), which is consistent with the numerical results. The optimal strategy for the RM against optimistic customers is to keep their belief private and set the service fee \(p\) at 2.764, yielding a revenue of 8.451.  For pessimistic customers in Table \ref{num:p1}, since \(1/\mathbb{E}[1/\Lambda] = 4.597 < \lambda\) and \(M(\xi^C) = 4.597 <\lambda \), Theorem \ref{thm:Rev_opt_P_U} indicates that \(\text{Rev}^P_{opt} < \text{Rev}^C_{opt}\), which also aligns with the numerical results. The optimal strategy for the RM against pessimistic customers is to disclose the true arrival rate to encourage customer participation and set the service fee \(p\) around 2.764, yielding a revenue of 7.639.

\section{Conclusion}

In this paper, we discuss the impact of information asymmetry on queueing strategies. Our theories are developed for the general distribution of customers' belief and can be applied to both discrete and continuous settings. This work can serve as an unobservable counterpart for the paper \cite{cui_blind_2016}, which focuses on the observable model. To the best of our knowledge, we are the first to analyze the information asymmetry in the unobservable case with general belief distribution and to answer how to make joint decisions over the service fee and information disclosure. We provide easy-to-follow guidelines for managers to make decisions on information disclosure, which depends largely on the comparison between the average arrival rate belief and true arrival rate. 

In our paper, we assume that the arrival rate is fixed within a single day and that managers have the opportunity to reset the service fee \(p\) each day. For future work, an interesting scenario would be one where managers have only a single opportunity to set pricing and must also make decisions under uncertainty. Another possible direction would be to consider Bayesian customers, who can update their beliefs upon observing the pricing or receiving signals from the managers.

\bibliography{references}  
\bibliographystyle{IEEEtran.bst}

\appendix 

\section{Analytical Properties of Relevant Functions} \label{Sec:Property}

\paragraph{Justification of differentiation under expectation.} According to the Leibniz integral rule, when the integration limits do not depend on $x$,
\[
\frac{d}{dx}\!\left(\int_{a}^{b} f(x,t)\,dt\right)
= \int_{a}^{b} \frac{\partial}{\partial x} f(x,t)\,dt.
\]
Since the expectation over $\Lambda$ can be written as an integral over its (bounded) support with fixed limits $[\lambda_{\min},\lambda_{\max}]$,
\[
\frac{d}{dx}\,\mathbb{E}[f(x,\Lambda)]
= \frac{d}{dx}\!\left(\int_{\lambda_{\min}}^{\lambda_{\max}} f(x,\lambda)\,dF_\Lambda(\lambda)\right)
= \int_{\lambda_{\min}}^{\lambda_{\max}} \partial_x f(x,\lambda)\,dF_\Lambda(\lambda)
= \mathbb{E}\!\left[\partial_x f(x,\Lambda)\right].
\] Thus, throughout this paper, it is always legitimate to differentiate under the expectation.
\subsection{Concavity and unique critical point of \(\text{Rev}^S(q)\)}\label{concave proof}
From \eqref{Rev^Sq}, we have
\[
\text{Rev}^S(q)=q\lambda\,\mathbb{E}\!\left[R-CW(q\Lambda)\right].
\]
Write $x=q\Lambda$. Then
\[
W(x)=\frac{1}{\mu}+\frac{s^{2}}{2}\,\frac{x}{1-x/\mu},\qquad
W'(x)=\frac{s^{2}}{2}\,\frac{1}{(1-x/\mu)^{2}},\qquad
W''(x)=\frac{s^{2}/\mu}{(1-x/\mu)^{3}},
\]
so $W',W''>0$ on the feasible set $\{q:\ q\Lambda/\mu<1\ \text{a.s.}\}$.

From \eqref{Rev^Sq},
\[
\text{Rev}^S(q)=\lambda q\,\mathbb{E}\!\left[R-C\,W(q\Lambda)\right],\qquad
(\text{Rev}^S)'(q)=\lambda\,\mathbb{E}\!\left[R-C\Big(W(q\Lambda)+q\Lambda W'(q\Lambda)\Big)\right],
\]
\[
(\text{Rev}^S)''(q)=-\,\lambda C\,\mathbb{E}\!\left[\Lambda W'(q\Lambda)+q\Lambda^{2}W''(q\Lambda)\right]
= -\,\lambda C\,\mathbb{E}\!\left[\frac{s^{2}}{2}\frac{\Lambda}{(1-q\Lambda/\mu)^{2}}
+ s^{2}\,\frac{q\Lambda^{2}}{\mu(1-q\Lambda/\mu)^{3}}\right] \;<\;0.
\]
Hence, $\text{Rev}^S$ is strictly concave on the feasible interval and has at most one critical point.

Moreover,
\[
(\text{Rev}^S)'(0)=\lambda\!\left(R-C\,W(0)\right)=\lambda\!\left(R-C/\mu\right)
\]
and
\[
(\text{Rev}^S)'(q)\;\le\;\lambda\!\left[R
- C\!\left( \frac{1}{\mu}+\frac{s^{2}}{2}\frac{q\lambda_{\min}}{1-q\lambda/\mu_{\min}}\right)\right]
\;\xrightarrow[q\uparrow (\,\lambda_{\min}/\mu\,)^{-}]{}\; -\infty.
\]
Thus, when $R\ge C/\mu$, there exists an interior root of $(\text{Rev}^S)'(\cdot)$; strict concavity implies this root (critical point) is unique.

\subsection{Unique root of \(\text{U}^S(q)\)}

From \eqref{u^s}, we have
\[
\text{U}^S(q)=R-p-\mathbb{E}\!\left[W(q\Lambda)\right].
\]
Differentiating under the expectation,
\[
(\text{U}^S)'(q)=-\,C\,\mathbb{E}\!\left[\Lambda\,W'(q\Lambda)\right]
=-\,C\,\mathbb{E}\!\left[\frac{s^{2}}{2}\,\frac{\Lambda}{(1-q\Lambda/\mu)^{2}}\right]\;<\;0,
\]

so $\text{U}^S(q)$ is strictly decreasing and thus has at most one root. When \(p = 0\), by an argument similar to that in (\ref{concave proof}), $\text{U}^S(q)$ has one unique root.

\subsection{Concavity and unique critical point of \(\text{Rev}^C(q)\), \(\text{SW}^C(q)\) and \(\text{SW}^S(q)\)}
From \eqref{Rev^Uq} and \eqref{SWSq}, we have
\[
\text{Rev}^C(q) = \text{SW}^C(q)=\text{SW}^S(q)=q\lambda\!\left[R-CW(q\Lambda)\right].
\]

\[
\text{Rev}^C(q)=\text{SW}^C(q)=\text{SW}^S(q)=q\lambda\Big[R-C\,W(q\lambda)\Big],\qquad W'(x)=\frac{s^{2}}{2}\frac{1}{(1-x/\mu)^{2}},\quad
W''(x)=\frac{s^{2}/\mu}{(1-x/\mu)^{3}}.
\]
By differentiation,
\[
(\text{SW}^C)'(q)=\lambda\Big[R-C\big(W(q\lambda)+q\lambda W'(q\lambda)\big)\Big],
\]
\[
(\text{SW}^C)''(q)=-\,\lambda C\Big[\lambda W'(q\lambda)+q\lambda^{2}W''(q\lambda)\Big]
= -\,\lambda C\left[\frac{\lambda s^{2}}{2(1-q\lambda/\mu)^{2}}
+ \frac{q\lambda^{2}s^{2}/\mu}{(1-q\lambda/\mu)^{3}}\right]\;<\;0.
\]

Hence, $\text{Rev}^C(q)$, $\text{SW}^C(q)$ and $\text{SW}^S(q)$ are strictly concave and therefore have at most one critical point. By an argument similar to that in (\ref{concave proof}), they have a unique critical point.

\section{Discrete Version of Theorem \ref{thm:Rev_U_P}} 
\label{Sec:Appendix}
Let the discrete arrival rate belief \( \Lambda \) take values in a finite set \( \{ \lambda_1, \lambda_2, \ldots, \lambda_k \} \), where \( \lambda_1 < \lambda_2 < \cdots < \lambda_k \). The probability mass function is denoted by \( p(\lambda_i) := \mathbb{P}(\Lambda = \lambda_i) \), for \( i = 1, \ldots, k \). The discrete version of Theorem \ref{thm:Rev_U_P} is as follows: 

\begin{theorem} 
Define the threshold function  \begin{align}
   M(\xi) &= 
   \begin{cases}
        1/\mathbb{E}(1/\Lambda)  & \text{if }\xi \leq  \lambda_{min} , \\
    \left( \sum_{\lambda_{\min} \leq \lambda_i \leq \xi} \dfrac{p(\lambda_i)}{\xi} 
      + \sum_{\xi < \lambda_i \leq \lambda_{\max}} \dfrac{p(\lambda_i)}{\lambda_i} \right)^{-1} 
      & \text{if } \lambda_{\min} < \xi < \lambda_{\max}.
   \end{cases}
   \end{align}
   Given $p$ fixed, In an M/G/1 queue,
   \begin{itemize}
       \item[(i)] When  $\xi\geq \lambda_{max}$,  \(\text{Rev}^P(p) = \text{Rev}^C(p)\).
        \item[(ii)] When  $\xi< \lambda_{max}$:
            \begin{itemize}
                \item[(a)] If $\lambda = M(\xi)$, then $\text{Rev}^P(p) = \text{Rev}^C(p)$; 
                \item[(b)]
                If $\lambda > M(\xi)$, then $\text{Rev}^P(p) > \text{Rev}^C(p)$; 
                \item[(c)]
                If $\lambda < M(\xi)$, then $\text{Rev}^P(p) < \text{Rev}^C(p)$.
                
            \end{itemize}
   \end{itemize}
   
   The threshold function $M(\xi)$ is continuous, monotonously increasing in terms of $\xi$ and piecewise convex.
\end{theorem}
\begin{proof}

(i) From \eqref{RevPp}, we have
$$\text{Rev}^P(p) = p\lambda \mathbb{E}[Q(p)],$$
where 
\begin{align} 
   \mathbb{E}[Q(p)] &= 
   \begin{cases}
        1   & \text{if }\xi\geq \lambda_{max} , \\
     \sum_{\lambda_{\min} \leq \lambda_i \leq \xi} p(\lambda_i)+\sum_{\xi < \lambda_i \leq \lambda_{\max}} \dfrac{\xi }{\lambda_i}p(\lambda_i)  & \text{if } \lambda_{min}<\xi< \lambda_{max}  ,\\
      \xi \mathbb{E}[1/\Lambda] & \text{if } \xi \leq  \lambda_{min}.
   \end{cases}
\end{align}

When $\xi\geq \lambda$, the threshold comes from comparing $\text{Rev}^P(p) = p\lambda$ and $\text{Rev}^C(p) = p\lambda$, which means that for any \( \xi\), \(\text{Rev}^P(p) = \text{Rev}^C(p)\). 

(ii) \begin{itemize}
\item[(ii.1)] When $\lambda \le \xi < \lambda_{\max}$, since $0<\mathbb{E}[Q(p)]<1$, we have
\[
\text{Rev}^C(p)=p\lambda \;>\; p\lambda\,\mathbb{E}[Q(p)] \;=\; \text{Rev}^P(p).
\]

\item[(ii.2)] When $\xi<\lambda$, the equality $\text{Rev}^P(p)=\text{Rev}^C(p)$ holds iff $\lambda=M(\xi)$.

\quad(a) If $\xi \le \lambda_{\min}$, comparing $\text{Rev}^P(p)=p\lambda\,\xi\,\mathbb{E}[1/\Lambda]$ with $\text{Rev}^C(p)=p\xi$ gives equality iff \( \lambda=\frac{1}{\mathbb{E}[1/\Lambda]}=M(\xi).\)

\quad(b) If $\lambda_{\min}<\xi<\lambda$, define
\[
  A(\xi) := \sum_{\lambda_{\min} \leq \lambda_i \leq \xi} \dfrac{p(\lambda_i)}{\xi} 
      + \sum_{\xi < \lambda_i \leq \lambda_{\max}} \dfrac{p(\lambda_i)}{\lambda_i}.
\]
Then $\text{Rev}^P(p)=\text{Rev}^C(p)$ iff
\[
\lambda=\frac{1}{A(\xi)}=M(\xi).
\]
Moreover, if $\lambda>M(\xi)$ then $\text{Rev}^P(p)>\text{Rev}^C(p)$, and if $\lambda<M(\xi)$ then $\text{Rev}^P(p)<\text{Rev}^C(p)$.
\end{itemize}

Finally, since $M(\xi)=\xi/\mathbb{E}[Q(p)]$ with $0<\mathbb{E}[Q(p)]<1$, we have $M(\xi)>\xi$ whenever $\xi<\lambda_{\max}$. Combining (ii.1) and (ii.2) yields the claim in (ii).

Next, we show the properties of the threshold function \( M(\xi) \). When $\lambda_{min}<\xi< \lambda_{max}$, we have
\[
A^{'}(\xi) = -\sum_{\lambda_{\min} \leq \lambda_i \leq \xi} \dfrac{p(\lambda_i)}{\xi^2} < 0,
\]
which implies
\[
M^{'}(\xi) = -\frac{A^{'}(\xi)}{A^2(\xi)} > 0.
\]
Furthermore,
\[
M^{''}(\xi) = -\frac{A^{''}(\xi)\,A(\xi) + 2\left[A^{'}(\xi)\right]^2}{A^3(\xi)},
\]
where
\[
A^{''}(\xi)= -\sum_{\lambda_{\min} \leq \lambda_i \leq \xi} \dfrac{2p(\lambda_i)}{\xi^3}
\]
Then, we have
    $$M^{''}(\xi) = \dfrac{B(\xi)}{A^3(\xi)},$$
  where 

\begin{align*}
B(\xi) = \sum_{\lambda_{\min} \leq \lambda_i \leq \xi} \dfrac{2p(\lambda_i)}{\xi^3}\sum_{\xi < \lambda_i \leq \lambda_{\max}} \dfrac{p(\lambda_i)}{\lambda_i}>0.
\end{align*}
Since \(M''(\xi)>0\) and \(M'(\xi)>0\), the threshold at each interval $\lambda_{i}<\xi< \lambda_{i+1}$ is monotonously increasing and convex.

Next, since \( \lim_{\xi \to \lambda_i^+} A(\xi) - \lim_{\xi \to \lambda_i^-} A(\xi) = 0 \), it follows that \( \lim_{\xi \to \lambda_i^+} M(\xi) - \lim_{\xi \to \lambda_i^-} M(\xi) = 0 \), showing that \( M(\xi) \) is also continuous at each point \( \lambda_i \), for \( i = 1, 2, \ldots, k \).

We can also verify that $M(\lambda_{min}) = 1/\mathbb{E}(1/\Lambda)$ and $M(\lambda_{max}) = \lambda_{max}$, showing that \(M(\xi)\) is monotonously increasing. Thus, \(M(\xi)\) is piece-wise continuous and monotonously increasing.

\end{proof}

\end{document}